\documentclass[10pt]{article} 
\usepackage[utf8]{inputenc}
\usepackage[T1]{fontenc}
\usepackage{appendix}
\usepackage{longtable}
\usepackage{amsmath}
\usepackage{amsthm}
\usepackage{amsfonts}
\usepackage{amssymb}
\usepackage{graphicx}
\usepackage{breqn}
\usepackage{mleftright}
\usepackage{xcolor}
\usepackage{abstract}
\usepackage{verbatim}
\usepackage[style=numeric, url=false]{biblatex}
\usepackage[left=2.5cm, right=2.5cm,bottom=3.5cm]{geometry}
\allowdisplaybreaks

\newtheorem{prop}{Proposition}[section]
\newtheorem{thm}[prop]{Theorem}
\newtheorem{lem}[prop]{Lemma}
\newtheorem{con}[prop]{Assumption}

\newtheorem{cor}[prop]{Corollary}
\theoremstyle{definition}
\newtheorem{defn}[prop]{Definition}
\newtheorem{rem}[prop]{Remark}
\numberwithin{equation}{section}

\renewcommand{\phi}{\varphi}
\renewcommand{\bar}[1]{\ensuremath{\overline{#1}}}
\renewcommand{\epsilon}{\varepsilon}
\newcommand{\tr}{\textnormal{tr}}
\newcommand{\trinorm}[1]{{\left\vert\kern-0.25ex\left\vert\kern-0.25ex\left\vert #1 
		\right\vert\kern-0.25ex\right\vert\kern-0.25ex\right\vert}}

\newcommand{\expl}[1]{\textcolor{gray}{#1}} % eplanation (not in final)
\newcommand{\delete}[1]{} %{\textcolor{gray}{#1}}
\newcommand*\samethanks[1][\value{footnote}]{\footnotemark[#1]}

\author{Francesca Biagini\thanks{Workgroup Financial and Insurance Mathematics, Department of Mathematics, Ludwig-Maximilians-Universit{\"a}t München, Theresienstrasse 39, 80333 Munich, Germany. Emails: biagini@math.lmu.de, bollweg@math.lmu.de.} \and Georg Bollweg\samethanks[1] \and Katharina Oberpriller\thanks{University of Freiburg, Ernst-Zermelo-Strasse 1, 79104 Freiburg, Germany. Email: katharina.oberpriller@stochastik.uni-freiburg.de.}}

\title{Non-linear Affine Processes with Jumps}

\begin{document}
	\maketitle
	
	% ABSTRACT
	\section*{\centering Abstract}
    \begin{center}
    	\begin{minipage}{12cm}
    		We present a probabilistic construction of $\mathbb{R}^d$-valued non-linear affine processes with jumps.
    		Given a set $\Theta$ of affine parameters, we define a family of sublinear expectations on the Skorokhod space under which the canonical process $X$ is a (sublinear) Markov process with a non-linear generator.
        	%	In particular, this allows us to interpret $X$ as an affine process under parameter uncertainty, as the generator $\mathcal{A}_x$ is given by the supremum of the generators of classical affine processes with characteristics in $\Theta(x)$.
    		This yields a tractable model for Knightian uncertainty for which the sublinear expectation of a Markovian functional can be calculated via a partial integro-differential equation.
    	\end{minipage}
    \end{center}
    \textbf{Keywords:} sublinear expectation,  non-linear affine processes, dynamic programming, PIDE\\
    \textbf{Mathematics Subject Classification (2020):} 60G65, 60G07 \\

	% SECTIONS
    \section{Introduction}	
	Aim of this paper is to introduce multi-dimensional non-linear affine processes with jumps.
	Classical affine processes are highly relevant for applications in finance, e.g. for modelling phenomena like volatility, stock prices, credit default or interest rates and have been studied in several contributions in the classical setting, see \cite{filipovic_time-inhomogeneous_2005}, \cite{kallsen_exponentially_2010}, \cite{keller-ressel_affine_2011}, \cite{cuchiero_affine_2011-1}, \cite{keller-ressel_regularity_2013}, \cite{keller-ressel_affine_2018}.
	% However, it is important to take into account model uncertainty by considering a family of possibly non-dominated priors instead of a fixed probability measure. 
	Over the last years, several different approaches have been developed in order to establish so-called robust settings, which are independent of the underlying priors, see among others \cite{acciaio_larsson_2017}, \cite{bbkn_2017}, \cite{denis_hu_peng_2010}, \cite{denis_martini_2006}, \cite{guo_pan_peng_2017}, \cite{neufeld_robust_2016}, \cite{nutz_constructing_2013} and \cite{soner_touzi_zhang_MRP}.
	
	To our knowledge, the first definition of affine processes under model uncertainty is from \cite{fadina_affine_2019}, where they consider the one-dimensional continuous case with canonical state space. Here, non-linear continuous affine processes are represented by a family of semimartingale laws, such that the differential characteristics are bounded from above and below by  affine functions depending on the current state. In particular, the coefficients of the affine functions are assumed to be in a parameter set $\Theta$ which is fixed. The idea in \cite{fadina_affine_2019} is based on the papers \cite{neufeld_nonlinear_2016} and \cite{hollender_levy-type_2016}, which consider non-linear Lévy processes (with jumps) and non-linear Markov processes, respectively. In all these works, a sublinear expectation associated to the underlying sets of priors is defined and a dynamic programming principle is derived. A different approach has been studied in \cite{denk_kupper_nendel_2020}, \cite{denk_roeckner_2020}, \cite{nendel_2020}, where non-linear semigroup theory is used to define non-linear Markov processes. \\
	The scope of this paper is to extend the results on non-linear affine processes in \cite{fadina_affine_2019} to include jumps and an arbitrary closed state space $\textnormal{S}\subseteq\mathbb{R}^d$ for a given set of parameters $\Theta$.
	This is achieved by combining the probabilistic construction from \cite{fadina_affine_2019} and \cite{neufeld_nonlinear_2016} with findings from \cite{hollender_levy-type_2016}. In particular, the latter paper provides sufficient conditions on the uncertainty sets such that the dynamic programming principle follows.
%	In \cite{hollender_levy-type_2016} they construct sublinear expectations from sets of probability measures, and establish sufficient conditions on the uncertainty sets such that the dynamic programming principle follows.
%	In \cite{fadina_affine_2019} and \cite{neufeld_nonlinear_2016} they study affine and Lévy processes under sublinear expectations by specifying a set of semimartingale laws via the characteristics of the canonical process, and then prove the dynamic programming principle and other important properties for the associated sublinear expectations.
	
	More specifically, we consider the canonical process $X$ on the Skorokhod space $\Omega=\textnormal{D}(\mathbb{R}_+,\mathbb{R}^d)$ equipped with a family of sublinear expectations $\{\mathcal{E}^x\}_{x\in\mathbb{R}^d}$ defined as
	\begin{equation}
	 \mathcal{E}^x(f):=\sup_{P\in\mathcal{P}_x} E_P[f],
	\end{equation}
	where $\mathcal{P}_x$ is a set of probability measures, and $E_P$ is the usual expectation with respect to $P$.
	In this setting, $\mathcal{P}_x$ is the set of priors such that under every $P\in\mathcal{P}_x$, $X$ is a semimartingale starting at $x$ with differential characteristics evolving in the set $\Theta(X)$. Here, by a slight abuse of notation $\Theta(\cdot)$ denotes an affine set-valued map for a fixed set of parameters $\Theta$.
	In line with \cite{neufeld_nonlinear_2016} and \cite{fadina_affine_2019}, we call this process \emph{non-linear affine process}.

    In this framework, it is enough to assume that the set of parameters $\Theta$ is closed to prove the dynamic programming principle and the Markov property for sublinear expectations. In contrast to the approach in \cite{fadina_affine_2019} and \cite{neufeld_nonlinear_2016}, our argument relies only on the separability of the space of Lévy measures.  
    %we prove the measurability of the graph of $x\mapsto\Theta(x)$, which immediately yields the dynamic programming principle and the Markov property for sublinear expectations using results from \cite{hollender_levy-type_2016}.
	%This result differs from the approach in \cite{fadina_affine_2019} and \cite{neufeld_nonlinear_2016} as it relies only on the separability of the space of Lévy measures and not on the specific form of the map $x\mapsto\Theta(x)$.
	Hence, it can be extended to other classes of processes, e.g. polynomial processes. 
	Moreover, since we consider a more general parameter set $\Theta$, we obtain a broader class of non-linear affine processes than in \cite{fadina_affine_2019}, even for the continuous case with dimension $d=1$ on the canonical state space.

	Furthermore, we slightly modify the construction in \cite{fadina_affine_2019} in order to admit an arbitrary closed state space $\textnormal{S}\subseteq\mathbb{R}^d$.
	%In \cite{fadina_affine_2019}, the map $x\mapsto \Theta(x)$ is independent of the choice of state space $\textnormal{S}$, and $\textnormal{S}$ can be chosen in accordance with $\Theta$ and the evolution of $X$.
	In particular, we fix the state space $\textnormal{S}$ a priori and use it to define the set-valued map $\Theta(X)$. In this way, we can guarantee that
	%This serves the purpose of “limiting” the evolution of $X$.
    the process $X$ is hindered from continuously exiting the state space unless it jumps outside of $\textnormal{S}$.
	These properties are formalised in Lemma \ref{lem:jump-exit} and Corollary \ref{cor:trivial_uncertainty_subset}.
	
	Finally, as in \cite{fadina_affine_2019} and \cite{neufeld_nonlinear_2016}, we establish a link to the corresponding parabolic partial integro-differential equation (PIDE)
	\begin{equation}\label{eq:intro-PIDE}
		\partial_t v(t,x)-\mathcal{A}_xv(t,x)=0,
	\end{equation}
	which can be understood as the analogue of the Kolmogorov equation.
	The non-linear generator $\mathcal{A}_x$ is the supremum of the generators of classical affine processes with parameters in $\Theta$.
	More specifically, we show that if the value function $v(t,x):=\mathcal{E}^x(\phi(X_t))$ is continuous for a non-linear affine proces $X$, then $v$ is a viscosity solution of the PIDE \eqref{eq:intro-PIDE} with initial condition $v(0,x)=\phi(x)$.
	The consideration of a general parameter set $\Theta$ and of a general state space $\textnormal{S}$ comes at the cost of many technical subtleties compared to the corresponding approaches in \cite{fadina_affine_2019} and \cite{neufeld_measurability_2014}.
	In particular, the PIDE \eqref{eq:intro-PIDE} does not satisfy the degenerate ellipticity condition in general. Thus most existing comparison and uniqueness results in the literature are not applicable.
	Hence, we cannot prove the uniqueness of a viscosity solution of the PIDE \eqref{eq:intro-PIDE} in full generality, as done in \cite{neufeld_nonlinear_2016}.
	However, we are able to establish a uniqueness result for a special case with $\textnormal{S}=\mathbb{R}^d$ and a modified parameter map $\hat{\Theta}$  with associated family of sublinear expectations $\lbrace \hat{\mathcal{E}}^x \rbrace_{x \in \mathbb{R}^d}$ and a non-linear generator $\hat{\mathcal{A}}_x$.
	In this case, if the value function $\hat{v}(t,x):=\hat{\mathcal{E}}^x(\phi(X_t))$ is continuous, then $v$ is the unique viscosity solution to the PIDE induced by $\hat{\mathcal{A}}_x$ with initial condition $\hat{v}(0,x)=\phi(x)$.
	This result is neither covered by the comparison results used in \cite{neufeld_nonlinear_2016} or \cite{fadina_affine_2019}, nor we assume a uniform bounded condition as in \cite{hollender_levy-type_2016} or \cite{kuhn_viscosity_2019}. 
	Instead, we only impose a Lipschitz condition to prove the uniqueness.\newline
	
	The remainder of this paper is organised as follows.
	In Section \ref{sec:setup}, we introduce the definition of a multidimensional non-linear affine process with jumps and state space $\textnormal{S}\subseteq \mathbb{R}^d$. In Section \ref{sec:Dynamic}, we prove the dynamic programming principle and the Markov property for the family of sublinear expectations.
	In Section \ref{sec:PIDE}, we establish the connection with the Kolmogorov-type backward equation. In Section \ref{sec:Uniqueness}, we give a uniqueness result for a special case.
	Finally, we present some examples in Section \ref{sec:examples}.
    
    \section{Non-linear affine processes} \label{sec:setup}
	
    For some $d\in\mathbb{N}$, let $\Omega=\text{D}(\mathbb{R}_{+},\mathbb{R}^d)$ be the space of all $\mathbb{R}^d$-valued càdlàg paths equipped with the Skorokhod topology and the corresponding Borel $\sigma$-algebra $\mathcal{F}$.
    Moreover, let $\mathbb{F}=(\mathcal{F}_t)_{t\geq 0}$ denote the raw filtration generated by the canonical process $X=(X_t)_{t\geq 0}$ given by $X_t(\omega)=\omega(t)$.
    In the sequel the $i$-th coordinate of the vector-valued process $X$ is denoted by $X^i$.
    Similarly, the $i,j$-th coordinate of a matrix-valued process $M$ is denoted by $M^{i,j}$.
    The same notation applies to elements in $\mathbb{R}^d$.
    
    Let $\mathfrak{P}(\Omega):=\mathfrak{P}(\Omega,\mathcal{F})$ denote the Polish space of all probability measures on $(\Omega,\mathcal{F})$. 
    Let $P \in \mathfrak{P}(\Omega)$. The process $X$ is called a $(P$,$\mathbb{F})$-semimartingale, if there exists right-continuous $\mathbb{F}$-adapted processes $M$ and $B$ such that $M_0=0=B_0$, $M$ is a local martingale, $B$ is of locally finite variation, and $X=X_0+M+B$ $P$-a.s.
    Fix a truncation function $h$, i.e. a bounded Borel-measurable function with $h(x)=x$ on some neighbourhood of zero, and let $\mu$ be the jump measure of $X$.
    Then using the unique canonical semimartingale decomposition as in \cite[Theorem II.2.34]{jacod_limit_2003}, we have $P$-a.s.
    \begin{align}\label{eq:CanonicalDecomposition}
    	X_t
    	&=	X_0
    	+\,^c\!M^P_t
    	+\underbrace{\int_{0}^t\int_{\mathbb{R}^d}h(z)\,(\mu-K^P)(ds,dz)}_{=:\,^j\!M^P_t}
    	+\underbrace{\int_0^t\int_{\mathbb{R}^d}h(z)\,K^P(ds,dz)}_{=:\,^j\!B^P_t}
    	+\,^c\!B^P_t
    	+\underbrace{\int_{0}^t \int_{\mathbb{R}^d} \left(z-h(z)\right)\,\mu(ds,dz)}_{=:J^P_t},
    \end{align}
    where $K^P$ is the predictable compensator of the jump measure $\mu$, $\,^c\!M^P$ and $\,^j\!M^P$ are the continuous and purely discontinuous local martingale parts respectively, $B^P:=\,^c\!B^P+\,^j\!B^P$ is the predictable finite variation part, and $J^P$ is the unbounded jump part.
    
    If $B^P, [\,^c\!M^P], K^P$ are absolutely continuous with respect to the Lebesgue measure $P$-a.s., there exists an $\mathbb{F}$-adapted process $(b^P,a^P,k^P)$ such that $P$-a.s.
    \begin{equation}
    	dB^P_t=b^P_t\,dt \qquad
    	d[\,^c\!M^P]_t=a^P_t\,dt \qquad
    	dK^P([0,t];dz)=k^P_t(dz)\,dt.
    \end{equation}
    Such a process $(b^P,a^P,k^P)$ is called \emph{$(P$,$\mathbb{F})$-differential semimartingale characteristics} of $X$ relative to $h$.
    
    Note that $a^P$ takes values in the cone of non-negative definite symmetric matrices $\mathbb{S}_+$, where $\mathbb{S}$ denotes the vector space of symmetric matrices.
    The process $k^P$ takes values in the cone of non-negative Lévy measures $\mathfrak{L}_+\subseteq \mathfrak{L}$, which are defined as
    \begin{align}\label{eq:LevyCone}
    	\mathfrak{L}_+&:=\left\{k\in\mathfrak{M}_+(\mathbb{R}^d)\,:\,\int_{\mathbb{R}^d}\left(\|z\|^2\wedge\, 1\right)\,k(dz)<\infty\text{ and }k\left(\{0\}\right)=0\right\},\\
    	\label{eq:LevySpace}
    	\mathfrak{L}&:=\left\{k\in\mathfrak{M}(\mathbb{R}^d)\,:\,\int_{\mathbb{R}^d}\left(\|z\|^2\wedge\, 1\right)\,|k(dz)|<\infty\text{ and }k\left(\{0\}\right)=0\right\},
    \end{align}
    where $\mathfrak{M}_{+}(\mathbb{R}^d)$ and $\mathfrak{M}(\mathbb{R}^d)$ are the cone of all non-negative measures and the vector space of all signed measures on $\mathbb{R}^d$, respectively.
    By \cite[Theorem 8.9.4]{bogachev_measure_2007}, the space of all finite measures $\mathfrak{M}^f(\mathbb{R}^d)$ 
    endowed with the Kantorovich-Rubinstein metric $d_{\mathfrak{M}^f}$ 
    is a separable metric space.
    A Lévy measure $k\in\mathfrak{L}$ can be associated with the finite measure $\hat{k}:\,A\mapsto\int_A(\|z\|^2\wedge\,1)\,k(dz)$, $A\in\mathcal{B}(\mathbb{R}^d)$.
    Thus, $\mathfrak{M}^f(\mathbb{R}^d)$ is a separable metric space with the metric $d_{\mathfrak{L}}(k,m):=d_{\mathfrak{M}^f}(\hat{k},\hat{m})$
    on 
    $\mathfrak{L}$, cf. \cite[Lemma 2.3]{neufeld_measurability_2014}. 
    We endow $\mathfrak{L}$ with the corresponding Borel $\sigma$-algebra.
    Moreover, we can define a metric $d_\pi$ such that $\mathbb{R}^d\times\mathbb{S}\times\mathfrak{L}$ is again a separable metric space (and so are subspaces and finite product spaces of it), cf. \cite[Theorem 16.4c]{willard_general_1970}.
    
    In classical stochastic analysis, an affine process is a Markov semimartingale with differential characteristics that are affine functions of the current state, i.e.
    \begin{equation}\label{eq:AffineFunctionsCharacteristics}
    	(b^P,a^P,k^P)=(\beta(X),\alpha(X),\nu(X)) \quad dt\otimes dP\text{-a.e.}
    \end{equation} for some affine functions $\beta, \alpha, \nu$ on the state space of $X$. Note that the functions $\beta, \alpha, \nu$ do not depend on $P$.
    %The affine functions $\beta, \alpha, \nu$ are often referred to as parameters of the affine process.
    In the following, we define an affine process under parameter uncertainty by allowing the differential characteristics $(b^P,a^P,k^P)$ to evolve in a random set $\Theta(X)$ which depends on the value of $X$ in an affine manner.
    For this purpose, fix a closed, non-empty state space $\textnormal{S}\subseteq\mathbb{R}^d$.
    Let $\beta=(\beta_0,\ldots,\beta_d)\in(\mathbb{R}^d)^{d+1}$, $\alpha=(\alpha_0,\ldots,\alpha_d)\in\mathbb{S}^{d+1}$, $(\nu_0,\ldots,\nu_d)\in\mathfrak{L}^{d+1}$.
    Define for $x=(x^1,\ldots,x^d)^T \in \mathbb{R}^d$ the following functions	
    \begin{align}
    	\beta(x)&:=\Big(\beta_0+(\beta_1,\ldots,\beta_d)\,x\Big)\,\mathbf{1}_{\textnormal{S}}(x)
    	=\Big(\beta_0+\sum_{i=1}^d x^i \beta^i \Big)\,\mathbf{1}_{\textnormal{S}}(x) \in \mathbb{R}^d, \label{eq:theta1}\\
    	\alpha(x)&:=\Big(\alpha_0+(\alpha_1,\ldots,\alpha_d)\,x\Big)\,\mathbf{1}_{\textnormal{S}}(x) 
    	=\Big(\alpha_0 + \sum_{i=1}^d x^i \alpha_i\Big)\,\mathbf{1}_{\textnormal{S}}(x) \in \mathbb{S},\label{eq:theta2}\\
    	\nu(x)&:=\Big(\nu_0+(\nu_1,\ldots,\nu_d)\,x\Big)\,\mathbf{1}_{\textnormal{S}}(x)
    	=\Big(\nu_0+\sum_{i=1}^d x^i \nu_i \Big)\,\mathbf{1}_{\textnormal{S}}(x) \in \mathfrak{L}. \label{eq:theta3}
    \end{align}
    From now on, we refer to any $\theta=(\beta,\alpha,\nu)
    %=(\beta_0,\ldots,\beta_d,\alpha_0,\ldots,\alpha_d,\nu_0,\ldots,\nu_d)
    \in(\mathbb{R}^d)^{d+1}\times\mathbb{S}^{d+1}\times\mathfrak{L}^{d+1}$ as \emph{parameter}, and identify it with the map
    \begin{equation}\label{eq:theta}
    	\mathbb{R}^d\rightarrow\mathbb{R}^d\times\mathbb{S}\times\mathfrak{L},\quad x\mapsto\theta(x):=(\beta(x),\alpha(x),\nu(x)),
    \end{equation}
    where $\alpha$, $\beta$, $\nu$ are as in \eqref{eq:theta1} - \eqref{eq:theta3}.
    Similarly for a subset $\Theta\subseteq
    (\mathbb{R}^d)^{d+1}\times\mathbb{S}^{d+1}\times\mathfrak{L}^{d+1}$ and $x\in\mathbb{R}^d$, define
    \begin{equation}\label{eq:Theta}
    	\Theta(x):=\Big\{\theta(x)\,:\,\theta\in\Theta\Big\}\subseteq\mathbb{R}^d\times\mathbb{S}\times\mathfrak{L}.
    \end{equation}
    A set $\Theta$ of parameters is \emph{closed}, if it is closed with respect to the topology which makes $(\mathbb{R}^d)^{d+1}\times\mathbb{S}^{d+1}\times\mathfrak{L}^{d+1}$ a separable metric space.
    %Note the implicit dependency on the state space $\textnormal{S}$ in \eqref{eq:theta} and \eqref{eq:Theta} via \eqref{eq:theta1} - \eqref{eq:theta3}.
    %This dependency, namely the indicator function, will be important for our construction. \\
    
    Finally, for $x \in \mathbb{R}^d$, define the following set of probability measures	
    \begin{equation}\label{eq:NLAJD}
        \mathcal{P}_x(\Theta):=\Big\{P\in\mathfrak{P}_{sem}^{ac}(\Omega)\,:\,\;P(X_0=x)=1;\:(b^P\!, a^P\!, k^P)\in\Theta(X)\text{ $dt\otimes dP$-a.e.}\Big\},
    \end{equation}
    with
    %where $\mathfrak{P}^{ac}_{sem}(\Omega)$ is the set of all semimartingale laws such that $X$ admits differential characteristics,
    \begin{equation}\label{eq:P-semac}
    	\mathfrak{P}_{sem}^{ac}(\Omega):=\Big\{P\in\mathfrak{P}_{sem}(\Omega)\,:\ \text{$X$ admits $(P$-$\mathbb{F})$-differential characteristics $(b^P,a^P,k^P)$} \Big\}
    \end{equation}
    and
    \begin{equation}
    	\mathfrak{P}_{sem}:= \Big\{P\in\mathfrak{P}(\Omega)\,:\ \text{$X$ is a semimartingale on  $(\Omega, \mathcal{F}, \mathbb{F},P)$} \Big\}.
    \end{equation}
    Note that $\mathfrak{P}_{sem}^{ac}$ is not empty, as it always contains the Dirac measures on $\Omega$ with point mass on a constant path.
    To each $\mathcal{P}_x(\Theta)$, we can associate the sublinear expectation
    \begin{equation}\label{eq:defSublinearExpectation}
    	\mathcal{E}^x(f):=\sup_{P\in\mathcal{P}_x(\Theta)}E_P[f]
    \end{equation}
    for any Borel-measurable function $f:\,\Omega\rightarrow\mathbb{R}$ with
    \begin{equation}
    	E_P[f ]:=E_P[f^+ ]-E_P[f^- ],
    \end{equation}
    where $f^+$ and $f^-$ denote the positive and negative parts of $f$ respectively.
    We use the convention $\sup \emptyset = - \infty$ and $\infty - \infty := - \infty.$
    	
    %Clearly, if $\mathcal{P}_x(\Theta)=\{P\}$ is a singleton, then the sublinear expectation is precisely the classical expectation with respect to $P$.
    %Otherwise, it can be interpreted as the greatest (i.e. worst or best) expectation of $f$ in the case of uncertainty in the probability measure (i.e. model).
    %The set $\mathcal{P}_x(\Theta)$ is thus often referred to as uncertainty set.
    	
    \begin{defn}\label{def:NLAP}
    	Let $\Theta\subseteq(\mathbb{R}^d)^{d+1}\times\mathbb{S}^{d+1}\times\mathfrak{L}^{d+1}$ be non-empty and closed. The tuple $(X,\{\mathcal{P}_x(\Theta)\}_{x\in\textnormal{S}},\textnormal{S})$, where $X$ is the canonical process and $\mathcal{P}_x(\Theta)$ is defined in \eqref{eq:NLAJD}, is called \emph{non-linear affine process} with parameter set $\Theta$ and state space $\textnormal{S}$.  
    \end{defn}
    	
    \begin{rem}\label{rem:filtration-closedness}
    	\begin{enumerate}
    		\item \label{rem:filtration}
    		Note that in \eqref{eq:P-semac} we consider $(P$,$\mathbb{F})$-semimartingales, where $\mathbb{F}$ is the raw filtration.
    		We could also work with $\mathbb{F}_+=(\mathcal{F}_{t+})_{t\geq 0}$ its right-continuous version or $\mathbb{F}_+^P=(\mathcal{F}_{t+}^P)_{t\geq 0}$ its usual augmentation under $P$.
    		By \cite[Proposition 2.2]{neufeld_measurability_2014}, the choice of filtration is not crucial.
    		In particular, the associated semimartingale characteristics are $P$-a.s. the same.
    		Hence, whenever the probability measure $P$ is fixed, we may consider $X$ as stochastic process on the augmented space $(\Omega,\mathcal{F}^P,\mathbb{F}^P_+,P)$ to avoid measurability issues.
    		\delete{\item \label{rem:closedness}
    			It is a priori not clear whether the defined sublinear expectations satisfy desirable properties.
    			For instance, we want to work with conditional expectations which satisfy the aggregation property as in the classical setting.
    			These properties are subsumed in the dynamic programming principle which we will establish in Section \ref{sec:Dynamic}.
    			For this, the closedness of $\Theta$ will be crucial.}
    		%\item 
    		%	Further note that the Definition \ref{def:NLAP} is in line with the definition of a Markov process in the classical setting in the sense that it includes {not only the process $X$ but also the state space $\textnormal{S}$, as well as} a family $\{\mathcal{P}_x(\Theta)\}_{x\in\textnormal{S}}$ of sets of probability measures ranging over the state space $\textnormal{S}$, see e.g. \cite{liggett_continuous_2010}, \cite{blumenthal_markov_1968}, \cite{cinlar_semimartingales_1980}, or \cite{duffie_affine_2003}.
    		%Further note that our definition of a non-linear affine process includes the process $X$ as well as the family of sets $\mathcal{P}_x(\Theta)$ ranging over all $x\in\textnormal{S}$.
    		%This is line with the definition of a Markov process in the classical setting, where there is a probability measure $P_x$ for every initial state $x\in\textnormal{S}$.
    		%As we will show later, a non-linear affine process behaves indeed like a Markov process under the associated sublinear expectation.
    		%Moreover, due to our definition in \eqref{eq:NLAJD}, the sets $\mathcal{P}_x(\Theta)$ are trivial for $x\notin\textnormal{S}$.
    		\item %Our aim is to incorporate model uncertainty by specifying sets of probability measures $\mathcal{P}_x\subseteq\mathfrak{P}(\Omega)$ via a set $\Theta\subseteq(\mathbb{R}^d)^{d+1}\times\mathbb{S}^{d+1}\times\mathfrak{L}^{d+1}$ of parameters, such that $X$ under the sublinear expectation
    	    %\begin{equation}
    	    %	\mathcal{E}^x:=\sup_{P\in\mathcal{P}_x}E_P
    	    %\end{equation}
    	    %can be interpreted as an affine process under parameter uncertainty.
        	Another approach for defining affine processes under parameter uncertainty is to consider the set of all probability measures $P$ such that $X$ is a $(P$,$\mathbb{F})$-semimartingale with differential characteristics $\theta(X)$ for some $\theta\in\Theta$.
        	That is, the family $\{\check{\mathcal{P}}_x(\Theta)\}_{x\in \textnormal{S}}$ given by
        	\begin{equation}\label{eq:parameter-uncertainty}
        		\check{\mathcal{P}}_x(\Theta)=\Big\{P\in\mathfrak{P}_{sem}^{ac}(\Omega)\,:\,P(X_0=x)=1;\:\exists\, {\theta\in\Theta}:\; \theta(X)=(b^P\!, a^P\!, k^P)\text{ $dt\otimes dP$-a.e.}\Big\}.
        	\end{equation}
    	\end{enumerate}
    \end{rem}
    
    \noindent We now study the behaviour outside of the state space $\textnormal{S}$ of the non-linear affine process $(X, \{\mathcal{P}_x(\Theta)\}_{x \in \textnormal{S}},\textnormal{S})$.
    For this purpose, we define the random times
    \begin{equation}\label{eq:exit-entrance-times}
    	\sigma:=\inf\left\{t> 0\,:\,X_t\notin\textnormal{S}\right\},
    	\qquad
    	\tau:=\inf\left\{t>\sigma\,:\,X_t\in\textnormal{S}\right\}.
    \end{equation}
    For every $x\in\mathbb{R}^d$, $P\in\mathcal{P}_x(\Theta)$, $\sigma$ and $\tau$ are $\mathbb{F}^P_+$-stopping times by the Début Theorem, cf. \cite[Theorem I.1.27]{jacod_limit_2003}.
    Note that, if $x\notin\textnormal{S}$, then $\sigma=0$ $P$-a.s. for all $P\in\mathcal{P}_x(\Theta)$.
    	
    \begin{lem}\label{lem:jump-exit}
    	For all $x\in\textnormal{S}$, $P\in\mathcal{P}_x(\Theta)$, the set
    	\begin{equation}
    		\left\{\sigma < \infty,\; X_\sigma\in\textnormal{S}\right\}
        \end{equation}
    	is $P$-null.
    \end{lem}
    
    \begin{proof}
    	Fix $x\in\textnormal{S}$ with $\mathcal{P}_x(\Theta)\neq\emptyset$, and let $P\in\mathcal{P}_x(\Theta)$.
    %	Since $(b^P,a^P,k^P)\in\Theta(X)$ $dt\otimes dP$-a.e., we can choose a version 
    	Then, for any $\epsilon >0$ we have $P$-a.s.
    	\begin{align}\label{eq:canonical-1}
    		\left( X_{\sigma+\epsilon}-X_\sigma \right)\,\mathbf{1}_{\{\sigma<\infty\}}
    		&=\bigg( \int_\sigma^{\sigma+\epsilon}\tilde{b}^P_s\,ds+\int_\sigma^{\sigma+\epsilon}\tilde{a}^P_s\,dW^P_s + \int_\sigma^{\sigma+\epsilon} \int_{\mathbb{R}^d} h(z) \Big(\mu(ds,dz)-\tilde{k}^P_s(dz)\,ds\Big) \nonumber \\
    		&\qquad  + \int_\sigma^{\sigma+\epsilon} \int_{\mathbb{R}^d} \left(z- h(z)\right) \,\mu(ds,dz) \bigg)\,\mathbf{1}_{\{\sigma<\infty\}}
    	\end{align}
    	for some $P$-Wiener process $W^P$ and for $(\tilde{b}^P,\tilde{a}^P,\tilde{k}^P)$ taking values in $\Theta(X)$ for all $(t,\omega)\in\mathbb{R}_+\times\Omega$.
    
    	Let $\omega\in \left\{\sigma < \infty,\; X_\sigma\in\textnormal{S}\right\}$,
    	then there exists an $\epsilon_*>0$ with $\sigma(\omega)+\epsilon_*<\tau(\omega)$.
    	Due to the right-continuity of $X(\omega)$, we can choose $\epsilon_*$ small enough such that $u\mapsto X_u(\omega)$ is continuous on $[\sigma(\omega),\sigma(\omega)+\epsilon_*]$.		
    	Since $\Theta(X)=\{(0,0,0)\}$ on $]\!]\sigma,\tau[\![$, we have in particular $(\tilde{b}^P_s(\omega),\tilde{a}^P_s(\omega),\tilde{k}^P_s(\omega))=(0,0,0)$ for $s\in]\sigma(\omega),\sigma(\omega)+\epsilon_*[$ by our choice of $\omega$ and $\epsilon_*$.
    	Moreover, we have $\mu([\sigma(\omega),\sigma(\omega)+\epsilon_*],\mathbb{R}^d)=0$ since the path is continuous on $[\sigma(\omega),\sigma(\omega)+\epsilon_*]$.
    	Hence, the right-hand side of \eqref{eq:canonical-1} is zero for $\omega$ and $\epsilon_*$.
    	Note that the left-hand side of \eqref{eq:canonical-1} is not zero since $X_{\sigma(\omega)}(\omega)\in\textnormal{S}$ and $X_{\sigma(\omega)+\epsilon_*}(\omega)\notin\textnormal{S}$.
    	Thus, $\omega$ is contained in the $P$-null set on which \eqref{eq:canonical-1} does not hold for $\epsilon=\epsilon_*$.
    	We conclude that
    	\begin{equation}
    		\left\{\sigma < \infty,\; X_\sigma\in\textnormal{S}\right\} \subseteq \bigcup_{n\in\mathbb{N}} \left\{ \eqref{eq:canonical-1} \text{ does not hold for }\epsilon=\frac{1}{n} \right\},
    	\end{equation}
    	which as countable union of $P$-null sets is again $P$-null.
    \end{proof}
    	
    \begin{lem}\label{lem:tau-infinite}
        For all $x\in\mathbb{R}^d$ and $P\in\mathcal{P}_x(\Theta)$, the set
        \begin{equation}
            \left\{\tau<\infty\right\}
        \end{equation}
        is $P$-null.
    \end{lem}
    	
    \begin{proof}
        Fix $x\in\mathbb{R}^d$ with $\mathcal{P}_x(\Theta)\neq\emptyset$, and let $P\in\mathcal{P}_x(\Theta)$. Note that $\sigma\leq\tau$, and $X_\tau\in\textnormal{S}$ on $\{\tau<\infty\}$.
        Thus,
        \begin{equation}
    	   \left\{\tau<\infty,\;X_\sigma=X_\tau\right\} \subseteq\{\sigma<\infty, X_\sigma\in\textnormal{S}\}.
        \end{equation}
    	If $x\in\textnormal{S}$, the right-hand side is $P$-null by Lemma \ref{lem:jump-exit}.
        Otherwise for $x\notin\textnormal{S}$, we have $X_\sigma=X_0=x\notin\textnormal{S}$ $P$-a.s., and the right-hand side is $P$-null.
    
        Now, consider the set $\{\tau<\infty,\;X_\tau\neq X_\sigma \}$.
        As in the proof of Lemma \ref{lem:jump-exit}, fix a process $(b,a,k)$ taking values in $\Theta(X)$ such that $P$-a.s.
    	    \begin{align}\label{eq:canonical-2}
    			\left(X_{\tau}-X_\sigma\right)\,\mathbf{1}_{\{\tau<\infty\}}
    			&=\bigg( \int_\sigma^{\tau}b(s)\,ds+\int_\sigma^{\tau}a(s)\,dW^P_s \nonumber \\
    			&\quad + \int_\sigma^{\tau} \int_{\mathbb{R}^d} h(z) \Big(\mu(ds,dz)-k(s;dz)\,ds\Big) + \int_\sigma^{\tau} \int_{\mathbb{R}^d} \left(z- h(z)\right) \,\mu(ds,dz) \bigg)\,\mathbf{1}_{\{\tau<\infty\}}
    		\end{align}
    		for some $P$-Wiener process $W^P$.
    		We have 
    		\begin{align*}
    		    E_P\left[ \int_\sigma^\tau\int_{\mathbb{R}^d}\left(\|z\|^2\wedge 1\right)\,\mu(ds,dz)\,\mathbf{1}_{\{\tau<\infty\}} \right]
    		    = E_P\left[ \int_\sigma^\tau\int_{\mathbb{R}^d}\left(\|z\|^2\wedge 1\right)\,k(s;dz)\,ds\,\mathbf{1}_{\{\tau<\infty\}} \right] 
    		    %&\leq E_P\left[ \underbrace{\int_{\mathbb{R}^d}\left(\|z\|^2\wedge 1\right)\,\left(k(\sigma;dz)+k(\tau;dz)\right)}_{<\infty}\,\underbrace{\lambda(\{\sigma,\tau\})}_{=0}\,\mathbf{1}_{\{\tau<\infty\}} \right]\\
    		    =0
    		\end{align*}
    		since $\Theta(X)=\{(0,0,0)\}$ on $]\!]\sigma,\tau[\![$ and the boundary $\lbrace \sigma, \tau \rbrace$ has Lebesgue measure zero. Consequently, the right-hand side of \eqref{eq:canonical-2} is zero $P$-a.s.
    		That is, $X_\tau=X_\sigma$ $P$-a.s. and $\left\{\tau<\infty,\;X_\sigma\neq X_\tau\right\}$ is $P$-null.
    		Hence, $\{\tau<\infty\}=\left\{\tau<\infty,\;X_\sigma=X_\tau\right\}\cup \left\{\tau<\infty,\;X_\sigma\neq X_\tau\right\}$ is also $P$-null.
    	\end{proof}
    	
    	As immediate consequence we obtain the following corollaries.
    	
    	\begin{cor}\label{cor:constant-outside}
    	    For all $x\in\mathbb{R}^d$ and $P\in\mathcal{P}_x(\Theta)$, $X$ is constant $P$-a.s. on $[\![\sigma,\infty[\![$.
    	\end{cor}
    	\begin{proof}
    	    This follows immediately from a close inspection of \eqref{eq:canonical-2} and Lemma \ref{lem:tau-infinite}.
    	\end{proof}
    	\begin{rem}
    	By Lemmas \ref{lem:jump-exit}, \ref{lem:tau-infinite} and Corollary \ref{cor:constant-outside} we can conclude if the process $X$ exits the set $\textnormal{S}$, it will never reenter $\textnormal{S}$.
    	Furthermore, the process $X$ can only exit the set $\textnormal{S}$ through a jump. 
    	\end{rem}
    	\begin{cor}\label{cor:trivial_uncertainty_subset}
    		Let $x\in\mathbb{R}^d\setminus\textnormal{S}$.
    		Then $\mathcal{P}_x(\Theta)\neq \emptyset$, and for all $P\in\mathcal{P}_x(\Theta)$ and $t\geq 0$, we have $X_t=x$ $P$-a.s.
    	\end{cor}
    	\begin{proof}
    	    Fix $x\in\mathbb{R}^d\setminus\textnormal{S}$.
    	    Let $\delta_x$ denote the Dirac measure on $\Omega$ with unit mass on the constant path $\omega\equiv x$.
    		Then clearly $\delta_x\in\mathcal{P}_x(\Theta)$, and hence $\mathcal{P}_x(\Theta)\neq \emptyset$.
    		The second statement follows from Corollary \ref{cor:constant-outside} by observing that $\sigma=0$ $P$-a.s. for all $P\in\mathcal{P}_x(\Theta)$.
    	\end{proof}
    
    	Corollary \ref{cor:trivial_uncertainty_subset} states that for all $x\notin\textnormal{S}$ and $P\in\mathcal{P}_x(\Theta)$ the process $X$ is a $P$-modification of the constant process with state $x$.
    	In particular, for $x\notin\textnormal{S}$, $t\geq0$, and any Borel-measurable function $f:\,\mathbb{R}^d\rightarrow{\mathbb{R}}$, we have
    	\begin{equation}
    		\mathcal{E}^x\left(f(X_t)\right)=\sup_{P\in\mathcal{P}_x(\Theta)} E_P[f(X_t)]=f(x).
    	\end{equation}
    	%Moreover, we can deduce that $X$ can only exit $\textnormal{S}$ through a jump.
    	
    	\begin{rem}\label{rem:non-empty}
    		Note that $\Theta\neq\emptyset$ does not immediately imply $\mathcal{P}_x(\Theta)\neq \emptyset$ for all $x\in\mathbb{R}^d$.
    		%For instance, take $\textnormal{S}=\mathbb{R}$ and $\Theta=\{(0,0,0,-1,0,0)\}\subseteq\mathbb{R}^2\times\mathbb{S}^2\times\mathfrak{L}^2$ (i.e. a negative volatility parameter), then $\Theta(x)\cap\mathbb{R}\times\mathbb{S}_+\times\mathfrak{L}_+=\emptyset$ for all $x>0$.
    		%Since the differential characteristics $(b^P,a^P,k^P)$ take values in $\mathbb{R}\times\mathbb{S}_+\times\mathfrak{L}_+$, we have $\mathcal{P}_x(\Theta)=\emptyset$ for $x>0$.
    		In order to avoid difficulties in that direction, we will henceforth assume that $\mathcal{P}_x(\Theta)\neq\emptyset$ for all $x\in\textnormal{S}$.
    		This assumption is not particularly strong as it follows immediately for many common choices of $\textnormal{S}$ and $\Theta$.
    		For instance, for the canonical state space $\textnormal{S}=\mathbb{R}_+^m\times\mathbb{R}^{d-m}$, we have a one-to-one correspondence between admissible parameters in the sense of \cite[Definition 2.6]{duffie_affine_2003} and affine processes due to \cite[Theorem 2.7]{duffie_affine_2003}.
    		More precisely, if $\Theta$ contains some admissible parameters $(\beta,\alpha,\nu)$, then there exists a family $\check{\mathcal{P}}_x(\Theta)$ of probability measures $\{P^x\}_{x\in \textnormal{S}}$ such that $X$ together with $\{P^x\}_{x\in\textnormal{S}}$ is a linear affine process with affine characteristics $(\beta,\alpha,\nu)$.
    		In particular, this implies $\emptyset\neq\check{\mathcal{P}}_x(\Theta) \subseteq\mathcal{P}_x(\Theta)$ for all $x\in\textnormal{S}$.
    		Clearly, by Corollary \ref{cor:trivial_uncertainty_subset}, $\mathcal{P}_x(\Theta)\neq\emptyset$ for all $x\notin\textnormal{S}$.
    \end{rem}  
    
    \section{Dynamic Programming Principle and Markov Property}\label{sec:Dynamic}	
	
    In this section, we prove that the family of subsets $\{\mathcal{P}_x(\Theta)\}_{x\in\textnormal{S}}$, or rather the associated sublinear expectations $\{\mathcal{E}^x\}_{x\in\textnormal{S}}$, is amenable to the dynamic programming principle as formalised in Proposition \ref{prop:DynamicProgramming}.
    In particular, we show that, for each initial value $x$, there exists a family of conditional sublinear operators $\{\mathcal{E}^x_t\}_{t\geq 0}$ that satisfies the dynamic programming property.
    As an immediate consequence we obtain a strong Markov property in Lemma \ref{lem:SubMarkov}.
    
    We build on results from \cite{nutz_constructing_2013}, which have been extended in \cite{hollender_levy-type_2016}.
    Thus, we first introduce the notation and the main results of \cite[Chapter 4]{hollender_levy-type_2016}, which we use later.
    
    Let $\tilde{\omega},\omega\in\Omega$, and $\tau:\Omega \to \mathbb{R}_+$ be a finite random time.
	The concatenation of paths $\tilde{\omega}$ and $\omega$ at $\tau$ is defined as
	\begin{equation}
		(\tilde{\omega}\otimes_{\tau}\omega)(t)= 	\tilde{\omega}(t)\,\mathbf{1}_{\left[0,\tau(\tilde{\omega})\right[\,}(t) +\Big(\,\tilde{\omega}\big(\tau(\tilde{\omega})\big)-\omega(0)+\omega\big(t-\tau(\tilde{\omega})\big)\Big)\,\mathbf{1}_{\left[\tau(\tilde{\omega}),\infty\right[\,}(t),\quad t\geq 0.
	\end{equation} 
	Given a function $f$ on $\Omega$, define the function $f^{\tau,\tilde{\omega}}$ on $\Omega$ by
	\begin{equation} 
		f^{\tau,\tilde{\omega}}(\omega):=f(\tilde{\omega}\otimes_\tau \omega).
	\end{equation}
	Let $P\in\mathfrak{P}(\Omega)$ and $\tau$ be a finite $\mathbb{F}$-stopping time.
	By \cite[Theorem 33.3]{billingsley_probability_1995}, there exists a family of regular conditional probability measures $\left\{P^\omega_\tau\right\}_{\omega\in\Omega}$ given $\mathcal{F}_\tau$, which can be chosen such that $P^\omega_\tau\in\mathfrak{P}(\Omega)$ is concentrated on $\omega\otimes_\tau \Omega$.
	More precisely, for any $A\in\mathcal{F}$, the map $\omega\mapsto P^\omega_\tau(A)$ is $\mathcal{F}_\tau$-measurable,
	\begin{equation}
		E_{P^\omega_\tau}\left[f\right]=E_P\left[f\,|\,\mathcal{F}_\tau\right](\omega)\quad\text{for $P$-a.e. }\omega\in\Omega
	\end{equation}
	when $f$ is bounded and $\mathcal{F}$-measurable, and
	\begin{equation}\label{eq:rcpd1}
		P^\omega_\tau\left(\Big\{\omega'\in\Omega\,:\,\omega=\omega'\text{ on }[0,\tau(\omega)]\Big\}\right)=1\quad\text{for all }\omega\in\Omega.
	\end{equation}
	Define the probability measure $P^{\tau,\omega}\in\mathfrak{P}(\Omega)$ by
	\begin{equation}\label{eq:probm-shift}
		P^{\tau,\omega}(A):=P^\omega_\tau(\omega\otimes_\tau A),\quad A\in\mathcal{F},
	\end{equation}
	then
	\begin{equation}
		E_{P^{\tau,\omega}}\left[f^{\tau,\omega}\right]=E_{P^\omega_\tau}\left[f\right]=E_P\left[f\,|\,\mathcal{F}_\tau\right](\omega)\quad\text{for $P$-a.e. }\omega\in\Omega.
	\end{equation}
	
	For $x\in\mathbb{R}^d$, let $\Omega_x\subseteq\Omega$ be the subspace of all paths starting at $x$.
	Moreover, let $\mathfrak{P}(\Omega_x)$ denote the space of all probability measures on $\Omega_x$.
	%We regard it as a subset of $\mathfrak{P}(\Omega)$.
	In the following, we consider a family $\{\mathcal{P}_x(t,\omega)\}_{(t,\omega)\in\mathbb{R}_+\times\Omega_x}$ of subsets of $\mathfrak{P}(\Omega_x)$ for a fixed $x\in\mathbb{R}^d$.
	\begin{con}\label{con:A}
	    For all $s\geq0$, finite $\mathbb{F}$-stopping times $s\leq \tau$, paths $\tilde{\omega}\in\Omega_x$, probability measures $P\in\mathcal{P}_x(s,\tilde{\omega})$ and $\sigma:=\tau^{s,\tilde{\omega}} -s$, the following holds.\newline
			\begin{longtable}{p{1cm}p{14cm}}
				\emph{(i)} & \emph{Adaptedness}: If $\tilde{\omega}=\omega$ on $[0,t]$, then $\mathcal{P}_x(t,\tilde{\omega})=\mathcal{P}_x(t,\omega)$;\\
				\emph{(ii)} & \emph{Measurability}: The graph
				\begin{equation}
						\left\{(P',\omega')\,:\,\omega'\in\Omega_x,\,P'\in\mathcal{P}_x(\tau(\omega),\omega')\right\}\subseteq\mathfrak{P}(\Omega_x)\times\Omega_x
				\end{equation} is analytic;\\
				\emph{(iii)} & \emph{Invariance}: $P^{\sigma,\omega}\in\mathcal{P}_x(\tau(\tilde{\omega}\otimes_s\omega),\tilde{\omega}\otimes_s\omega)$ for $P$-a.e. $\omega\in\Omega_x$;\\
				\emph{(iv)} & \emph{Stability}: If $\nu:\,\Omega_x\rightarrow\mathfrak{P}(\Omega_x)$ is an $\mathcal{F}_{\sigma}$-measurable kernel with
				\begin{equation}
					\nu(\omega)\in\mathcal{P}_x(\tau(\tilde{\omega}\otimes_s\omega),\tilde{\omega}\otimes_s\omega)\qquad\text{for $P$-a.e. $\omega\in\Omega_x$,}
				\end{equation} then 
				the measure $\bar{P}$ defined by
				\begin{equation}
					\bar{P}(A):=\int_{\Omega_{x}}\int_{\Omega_{x}} (\mathbf{1}_A)^{\sigma,\omega}(\omega')\,\nu(\omega;d\omega')\,P(d\omega)\qquad\text{for all $A\in\mathcal{F}$}
				\end{equation}
				is an element of $\mathcal{P}_x(s,\tilde{\omega})$.
			\end{longtable} 
	\end{con}
	If Condition \ref{con:A} holds, we set $\mathcal{P}_x:=\mathcal{P}_x(0,\omega)$ since $\mathcal{P}_x(0,\omega)$ is independent of $\omega$ due to Condition \ref{con:A} (i).
	We now state the relevant results from \cite{hollender_levy-type_2016}, cf. Theorem 4.29 and Lemma 4.30.
	
	\begin{thm}\label{thm:Hollender4.29}
		Suppose that $\{\mathcal{P}_x(t,\omega)\}_{(t,\omega)\in\mathbb{R}_+\times\Omega_x}$ is a family which satisfies Condition \ref{con:A}.
		Then for any finite $\mathbb{F}$-stopping time $\tau$ and upper semianalytic function $f:\,\Omega_x\rightarrow\bar{\mathbb{R}}$, the function
		\begin{equation}\label{eq:hollender-sublinear-expectation}
		     \mathcal{E}^x_\tau(f)(\omega):= \sup_{P\in\mathcal{P}_x(\tau(\omega),\omega)}E_P\left[f^{\tau,\omega}\right]
		\end{equation}
		is upper semianalytic and $\mathcal{F}_\tau^*$-measurable.
		Moreover, for finite stopping $\mathbb{F}$-stopping times $\sigma<\tau$, it holds
		\begin{equation} \label{eq:Aggregation}
			\mathcal{E}^x_\sigma 	\left(\mathcal{E}^x_\tau(f)\right)(\omega)=\mathcal{E}^x_\sigma(f)(\omega)
		\end{equation}
	    for all $\omega\in\Omega_x$.
		
    	Further, if for each $x\in\mathbb{R}^d$, the family $\{\mathcal{P}_x(t,\omega)\}_{(t,\omega)\in\mathbb{R}_+\times\Omega_x}$ 
		satisfies Condition \ref{con:A}., and the graph of $x\mapsto\mathcal{P}_x$, i.e., the set
		\begin{equation} \label{eq:SetAnalytic}
			\Big\{(x,P)\in\mathbb{R}^d\times\mathfrak{P}(\Omega)\,:\,x\in\mathbb{R}^d\text{ and }P|_{\Omega_x}\in\mathcal{P}_x\Big\},
		\end{equation}
		is analytic, then $x\mapsto\mathcal{E}_0^x(f):=\mathcal{E}_0^x(f|_{\Omega_x})$, where $\mathcal{E}^x_0$ is defined as in \eqref{eq:hollender-sublinear-expectation}, is upper semianalytic for every upper semianalytic function $f.$
	\end{thm}
	
	For $s\geq 0$, let $\vartheta_s:\,\Omega\rightarrow\Omega$ denote the time shift operator defined by
	\begin{equation} \label{eq:TimeShift}
		\vartheta_s(\omega)(t):=\omega(s+t),\quad t\geq 0.
	\end{equation} 
	We now summarise Proposition 4.31 and Lemma 4.32 in \cite{hollender_levy-type_2016}.
	\begin{prop}\label{prop:Hollender4.31}
		Suppose that the map
		\begin{equation}
			H\,:\mathbb{R}_+\times\Omega\rightarrow 	2^{\mathbb{R}^d\times\mathbb{S}_+\times\mathfrak{L}_+}
		\end{equation}
		is such that
		\begin{equation} \label{eq:AdaptedFiltration}
			\left\{(t,\omega,h)\in [0,T]\times\Omega\times(\mathbb{R}^d\times\mathbb{S}_+\times\mathfrak{L}_+)\,:\,h\in H(t,\omega)\right\}\in\mathcal{B}([0,T])\otimes\mathcal{F}_T\otimes\mathcal{B}(\mathbb{R}^d\times\mathbb{S}_+\times\mathfrak{L}_+) 
		\end{equation}
		for all $T\geq 0$.
		Then for all $x\in\mathbb{R}^d$, the family $\{\mathcal{P}^H_x(s,\tilde{\omega})\}_{(s,\tilde{\omega})\in\mathbb{R}_+\times\Omega_x}$ defined by
		\begin{equation}\label{eq:P-H}
			\mathcal{P}^H_x(s,\tilde{\omega}):=\Big\{P\in\mathfrak{P}_{sem}^{ac}(\Omega_x)\,:\,(b^P,a^P,k^P)\in H(s+\cdot,\tilde{\omega}\otimes_s\cdot) \quad dt\otimes P\text{-a.e.}\Big\}
		\end{equation}
		satisfies all assumptions from Condition \ref{con:A} and the set in \eqref{eq:SetAnalytic} is analytic.\\
		In particular, \eqref{eq:AdaptedFiltration} holds for all $T\geq 0$ if $H(t,\omega):=\bar{H}(\omega(t))$, where $\bar{H}:\,\mathbb{R}^d\rightarrow 2^{\mathbb{R}^d\times\mathbb{S}_+\times\mathfrak{L}_+}$ is a map with a Borel-measurable graph, i.e., the set
		\begin{equation}
			\Big\{(x,h)\in\mathbb{R}^d\times(\mathbb{R}^d\times\mathbb{S}_+\times\mathfrak{L}_+)\,:\,h\in \bar{H}(x)\Big\}
		\end{equation}
		is Borel.
		In this case, the corresponding sublinear conditional expectation from Theorem \ref{thm:Hollender4.29} satisfies
		\begin{equation}
			\mathcal{E}_0^x\left(\mathcal{E}_0^{X_t}(f)\right)=\mathcal{E}_0^x(f\circ\vartheta_t)
		\end{equation}
		for every upper semianalytic function $f:\,\Omega\rightarrow\bar{\mathbb{R}}$, $t\geq 0$ and $x\in\mathbb{R}^d$, where $\mathcal{E}_0^{X_t}$ is defined as \begin{equation*}
		    \mathcal{E}_0^{X_t}(\cdot)=\mathcal{E}_0^x(\cdot)\vert_{x=X_t}.
		\end{equation*}
	\end{prop}
	
	The results from \cite{hollender_levy-type_2016} provide a method for constructing a family of subsets that satifisfies Condition \ref{con:A}.
	In particular, by Proposition \ref{prop:Hollender4.31} it suffices to prove that the graph of $x\mapsto\Theta(x)$ is Borel-measurable.
	
	\begin{lem}\label{lem:Borel}
		Let $\Theta\subseteq(\mathbb{R}^d)^{d+1}\times\mathbb{S}^{d+1}\times\mathfrak{L}^{d+1}$ be non-empty and closed.
		Then the set
		\begin{equation}
			\left\{(x,b,a,k)\in\mathbb{R}^d\times\mathbb{R}^d\times\mathbb{S}_+\times\mathfrak{L}_+\,:\,(b,a,k)\in\Theta(x)\right\}
		\end{equation}
		is Borel, {where $\Theta(x)$ is defined in \eqref{eq:Theta}.}
	\end{lem}
	\begin{proof}
		Recall that $\Theta(x)=\{\theta(x):\,\theta\in\Theta\}$, and every $\theta \in\Theta$ defines a Borel-measurable map  $x\mapsto\theta(x)$, cf. \eqref{eq:theta}.
		For each $\theta \in \Theta$, consider the Borel-measurable map $d_{\tau}: \mathbb{R}^d\times\mathbb{R}^d\times\mathbb{S}_+\times\mathfrak{L}_+\rightarrow\mathbb{R}_+$ given by
		\begin{equation}
			d_{\tau}(x,b,a,k):= d_{\pi}\Big((b,a,k),\theta(x)\Big),
			% d_\pi is jointly continuous (thus jointly measurable). composition of measurable maps are measurable				
		\end{equation}
		where {$\mathfrak{L}_+$ is defined in \eqref{eq:LevyCone}} and $d_{\pi}$ denotes a metric on $\mathbb{R}^d\times\mathbb{S}\times\mathfrak{L}$.
		As a subspace of a separable metric space, $\Theta\subseteq(\mathbb{R}^d)^{d+1}\times\mathbb{S}^{d+1}\times\mathfrak{L}^{d+1}$ is separable, i.e., there exists a countable, dense subset $\Theta_c$ in $\Theta$.
		Fix such a subset $\Theta_c$, and consider
		\begin{equation}
			\inf_{\theta\in\Theta_c} d_{\pi}\Big((b,a,k),\theta(x)\Big).
			% countable infima of measurable maps are measurable
		\end{equation}
		The pointwise infimum of countably many Borel-measurable maps is again Borel-measurable.
		Hence, the set
		\begin{equation}
			G:=\left\{(x,b,a,k)\in\mathbb{R}^d\times\mathbb{R}^d\times\mathbb{S}_+\times\mathfrak{L}_+\,:\,\inf_{\theta\in\Theta_c} d_{\pi}\Big((b,a,k),\theta(x)\Big)= 0\right\}
		\end{equation}
		is Borel.
		Further, observe that
		\begin{equation}
			G=\left\{(x,b,a,k)\in\mathbb{R}^d\times\mathbb{R}^d\times\mathbb{S}_+\times\mathfrak{L}_+\,:\,(b,a,k)\in\Theta(x)\right\}.
		\end{equation}
		The $\supseteq$-inclusion is trivial.
		For the other direction, fix $(x,b,a,k)\in G$.
		Then there exists a sequence $(\theta_n)_{n\in\mathbb{N}}$ in $\Theta_c$ such that
		\begin{equation}
			\lim_{n\rightarrow\infty}d_\pi \Big((b,a,k),\theta_n(x)\Big)=0.
		\end{equation}
		Since $\bar{\Theta_c}=\Theta$, we have $\theta:=\lim_{n\rightarrow\infty}\theta_n\in\Theta$. Moreover, by the continuity of $d_\pi$ we have
		\begin{equation}
			d_\pi\Big((b,a,k),\theta(x)\Big)=d_\pi\Big((b,a,k),\lim_{n\rightarrow\infty}\theta_n(x)\Big)=\lim_{n\rightarrow\infty}d_\pi \Big((b,a,k),\theta_n\left(x\right)\Big)=0.
		\end{equation}
		Hence, $(b,a,k)=\theta(x)\in\Theta(x)$, and the identity holds.
	\end{proof}
	
	\begin{rem}\label{rem:Borel}
		The proof is stated for affine parameters, but solely relies on the separability of the space $\mathbb{R}^d\times\mathbb{S}\times\mathfrak{L}$.
		Thus, it also carries over for polynomial processes, i.e., when we consider polynomial functions in \eqref{eq:theta1}-\eqref{eq:theta3}. In particular, this implies that non-linear polynomial processes can be defined analogously and satisfy the dynamic programming principle. We refer to e.g. \cite{cuchiero_affine_2011} for an introduction to classical polynomial processes.
	\end{rem}
	
	\begin{prop}\label{prop:DynamicProgramming}
		Let $\Theta\subseteq(\mathbb{R}^d)^{d+1}\times\mathbb{S}^{d+1}\times\mathfrak{L}^{d+1}$ be non-empty and closed.
		For $x\in\mathbb{R}^d$, let the family $\{\mathcal{P}_x(s,\omega)\}_{(s,\omega)\in\mathbb{R}_+\times\Omega}$ be defined by
		\begin{equation}
		    \mathcal{P}_x(s,\omega):=\Big\{P\in\mathfrak{P}_{sem}^{ac}(\Omega)\,:\,X_0=x\;P\text{-a.s.},\quad (b^P,a^P,k^P)\in\Theta(X_{s+\cdot} (\omega\otimes_s\cdot ))\text{ }dt\otimes dP\text{-a.e.}\Big\}
		\end{equation}
		for $\omega\in\Omega_x$, and $\mathcal{P}_x(s,\omega):=\emptyset$ otherwise.
		Then, for any upper semianalytic function $f:\,\Omega\rightarrow\bar{\mathbb{R}}$ and finite $\mathbb{F}$-stopping times $\sigma\leq\tau$, the function
		\begin{equation}
			\mathcal{E}^x_{\tau}(f)(\omega):=\sup_{P\in\mathcal{P}_x(\tau(\omega),\omega)} E_P\left[f^{\tau,\omega}\right]
		\end{equation}
		is upper semianalytic and $\mathcal{F}_\tau^*$-measurable, and the dynamic programming principle holds, i.e.,
		\begin{equation}\label{eq:aggregation_property}
			\mathcal{E}^x_\sigma\left(\mathcal{E}^x_\tau(f)\right)=\mathcal{E}^x_\sigma(f)\quad\text{on }\Omega.
		\end{equation}
	\end{prop}
	
	\begin{proof}
		The proof is a straightforward application of Theorem \ref{thm:Hollender4.29} and Proposition \ref{prop:Hollender4.31}.
		%Let $\Theta\subseteq (\mathbb{R}^d)^{d+1}\times\mathbb{S}^{d+1}\times\mathfrak{L}^{d+1}$ be non-empty and closed.
		%Moreover let $f:\,\Omega\rightarrow\bar{\mathbb{R}}$ be upper semianalytic and $\sigma<\tau$ be finite $\mathbb{F}$-stopping times.
		As an immediate consequence of Lemma \ref{lem:Borel}, the graph of the map $\bar{H}:\,\mathbb{R}^d\rightarrow 2^{\mathbb{R}^d\times\mathbb{S}_+\times\mathfrak{L}_+}$ given by
		\begin{equation}\label{eq:H_map}
			 \bar{H}(x)=\Theta(x)\cap(\mathbb{R}^d\times\mathbb{S}_+\times\mathfrak{L}_+)
		\end{equation}
		is Borel since $\Theta$ is non-empty and closed.
		Fix $x\in\mathbb{R}^d$, and consider the map $H:\,(t,\omega)\mapsto\bar{H}(\omega(t))=\bar{H}(X_t(\omega))$.
        Let $\mathcal{P}^H_x(s,\omega)$ with $s\geq 0$ and $\omega\in\Omega_x$ be defined as in Proposition \ref{prop:Hollender4.31}. Then
        \begin{align}
            \mathcal{P}^H_x(s,\omega)
            :&=\Big\{P\in\mathfrak{P}_{sem}^{ac}(\Omega_x)\,:\,(b^P,a^P,k^P)\in H(s+\cdot,\omega\otimes_s\cdot )\text{ }dt\otimes dP\text{-a.e.}\Big\} \nonumber \\
            &= \Big\{P\in\mathfrak{P}_{sem}^{ac}(\Omega_x)\,:\,(b^P,a^P,k^P)\in \bar{H}\left(X_{s+\cdot}(\omega\otimes_s\cdot)\right)\text{ }dt\otimes dP\text{-a.e.}\Big\} 
            \nonumber\\
            &= \Big\{P\in\mathfrak{P}_{sem}^{ac}(\Omega_x)\,:\,(b^P,a^P,k^P)\in \Theta\left(X_{s+\cdot}(\omega\otimes_s\cdot)\right)\text{ }dt\otimes dP\text{-a.e.}\Big\},
        \end{align}
        where the last identity is immediate by  \eqref{eq:H_map}  since the characteristics $(b^P,a^P,k^P)$ takes values in $\mathbb{R}^d\times\mathbb{S}_+\times\mathfrak{L}_+$.
        Observe that for all $\omega\in\Omega_x$, we have
        \begin{equation}
            \mathcal{P}_x^H(s,\omega)=\Big\{ P|_{\Omega_x}\,:\,P\in\mathcal{P}_x(s,\omega)\Big\},
        \end{equation}
        and thus
        \begin{equation}
            \mathcal{E}^x_\tau(f)(\omega)=\sup_{P\in\mathcal{P}_x(\tau(\omega),\omega)} E_P\left[f^{\tau,\omega}\right]=\sup_{P\in\mathcal{P}^H_x(\tau(\omega),\omega)} E_P\left[(f|_{\Omega_x})^{\tau,\omega}\right].
        \end{equation}
        Since $f|_{\Omega_x}$ is upper semianalytic for an upper semianalytic $f:\,\Omega\rightarrow\bar{\mathbb{R}}$, Proposition \ref{prop:Hollender4.31} and Theorem \ref{thm:Hollender4.29} yields the desired properties on $\Omega_x$.
        For $\omega\in\Omega\setminus\Omega_x$, we have
        \begin{equation}
            \mathcal{E}^x_\tau(f)(\omega)=-\infty,
        \end{equation}
        thus the desired measurability properties and the dynamic programming property on $\Omega$ follow immediately.
    \end{proof}
        
    Note that the conditional sublinear expectations from Proposition \ref{prop:DynamicProgramming} are compatible with the sublinear expectations $\mathcal{E}^x$ defined in \eqref{eq:defSublinearExpectation} in the sense that for all $x\in\mathbb{R}^d$ and upper semianalytic $f:\,\Omega\rightarrow\bar{\mathbb{R}}$, it holds that
    \begin{equation}
        \mathcal{E}^x(f)=\mathcal{E}^x_0(f)\quad\text{on }\Omega_x.
    \end{equation}
    
	In the classical setting, the Markov property is quintessential to affine processes.
	The dynamic programming principle of Proposition \ref{prop:DynamicProgramming} yields an analogue for sublinear conditional expectations.

	\begin{cor}\label{lem:SubMarkov}
		Let $\Theta\subseteq(\mathbb{R}^d)^{d+1}\times\mathbb{S}^{d+1}\times\mathfrak{L}^{d+1}$ be non-empty and closed.
		Then for any $x \in \mathbb{R}$, $s,t\geq 0$ and any upper semianalytic function $f:\,\Omega\rightarrow\bar{\mathbb{R}}$, we have
		\begin{equation}
			\mathcal{E}^x_s(f\circ\vartheta_s)=\mathcal{E}^{X_s}(f)
	        %:=\mathcal{E}^y(f)\big|_{y=X_s}
	        \quad \text{on }\Omega_x,
		\end{equation}
		where $\vartheta_s$ is the time shift defined in \eqref{eq:TimeShift}.
		In particular, 
		\begin{equation}\label{eq:markov}
			\mathcal{E}_s^x(f(X_{t+s}))
			=\mathcal{E}^{X_s}(f(X_t))\quad\text{on }\Omega_x.
		\end{equation}
	\end{cor}
	
	\begin{proof}
		It follows immediately from Proposition \ref{prop:Hollender4.31}.
	\end{proof}
    
    \section{Kolmogorov PIDE under model uncertainty}\label{sec:PIDE}
	In the classical setting, an affine process $X=(X_t)_{t\geq 0}$ with characteristics $(\beta,\alpha,\nu)$ is a Markov process with generator
	\begin{align}
		A_x^{(\beta,\alpha,\nu)}f(x)
		& =
		D_x f(x)^T\,\beta(x) +\frac{1}{2}\, \tr\Big[D_{x}^2f(x)\,\alpha(x)\Big]
		%\nonumber\\
		%&\quad
		+\int_{\mathbb{R}^d}\left[f(x+z)-f(x)-D_xf(x)^T\,h(z)\right]\nu(x;dz).
	\end{align}
	%where $D_x f(x)=(\partial_{x^1}f(x),\ldots,\partial_{x^d})^T f(x)$ is the gradient of $f$ at $x$ and $D_x^2f(x)=(\partial_{x^i}\partial_{x^j}f(x))_{1\leq i,j\leq d}$ is the Hessian matrix of $f$ at $x$.
	The generator induces a PIDE, usually referred to as evolution or Kolmogorov equation given by
	\begin{equation}\label{eq:classicalPIDE}
		\partial_t v(t,x)-A_x^{(\beta, \alpha, \nu)}v(t,x)=0.
	\end{equation}
	Under suitable conditions on $\phi$, the value function $v(t,x):=E^x[\phi(X_t)]$ is a viscosity solution of  \eqref{eq:classicalPIDE} satisfying the initial condition $v(0,x)=\phi(x)$.
	The goal of this section is to establish a similar result for non-linear affine processes.
	More precisely, for a bounded, Lipschitz continuous function $\phi \in \text{Lip}_b(\mathbb{R}^d)$, a non-linear affine process $X$ with parameter set $\Theta$ and state space $\textnormal{S}$, we show that if the value function $v:\;\mathbb{R}_+\times\mathbb{R}^d\rightarrow\bar{\mathbb{R}}$ given by
	\begin{equation}\label{eq:def-value-function}
	v(t,x):=\mathcal{E}^x(\phi(X_t))
	\end{equation}
	is continuous, then $v$ is a viscosity solution of the PIDE 
	\begin{align}\label{eq:PIDE}
		\partial_t v(t,x)-\mathcal{A}_xv(t,x)&=0\\
		\label{eq:PIDE-initial} v(0,x)&=\phi(x),
	\end{align}
	where the operator $\mathcal{A}_y$ is defined by
	\begin{align}\label{eq:def-nonlinear-A}
		\mathcal{A}_y f(x)
		&:=
		%\sup_{(\alpha, \beta, \nu) \in \Theta} A_y^{(\beta,\alpha,\nu)} f(x) \nonumber\\
		%& =
		\sup_{(\beta,\alpha,\nu)\in\Theta}\Bigg\{D_x f(x)^T\,\beta(y)+\frac{1}{2}\,\tr\Big[D_{x}^2 f(x)\,\alpha(y)\Big]
		%\nonumber\\
		%&\qquad\qquad\qquad
		+\int_{\mathbb{R}^d} \left[f(x+z)-f(x)-D_x f(x)^T\,h(z)\right]\,\nu(y;dz) \Bigg\}.
	\end{align}
	To be specific, we use the following definition of a viscosity solution, as in \cite{neufeld_nonlinear_2016}.
	\begin{defn}[Viscosity Solution]
		%Let $\textnormal{D}\subseteq\mathbb{R}^d$ be open.
		Let $G:\,\mathbb{R}^d\times\mathbb{R}^d\times\mathbb{S}\times\textnormal{C}^2_b(\mathbb{R}^d)\rightarrow\mathbb{R}$ be some operator.
		An upper semicontinuous function $v:\,\mathbb{R}_+\times\mathbb{R}^d\rightarrow \bar{\mathbb{R}}$ is called \emph{viscosity subsolution} of
		\begin{equation}\label{eq:def-PIDE}
			\partial_t v(t,x)+G(x,D_xv(t,x),D_x^2v(t,x),v(t,\cdot)) = 0
		\end{equation} in a domain $D\subseteq\mathbb{R}_+\times\mathbb{R}^d$ if for any interior point $(t,x)\in D^o$ and any $\psi\in\textnormal{C}^{2,3}_b(\mathbb{R}_+\times\mathbb{R}^d)$ with $\psi(t,x)=v(t,x)$ and $v\leq \psi$ on $\mathbb{R}_+\times\mathbb{R}^d$, we have
		\begin{equation}\label{eq:def-visc2}
			\partial_t\psi(t,x)+G(x,D_xv(t,x),D_x^2v(t,x),v(t,\cdot))\leq 0.
		\end{equation}
		The definition of a \emph{viscosity supersolution} of \eqref{eq:def-PIDE} in $D$ is obtained by reversing the inequalities and considering a lower semicontinuous function $v$. 
		A function $v$ is called \emph{viscosity solution} of \eqref{eq:def-PIDE} in $D$ if it is both a viscosity sub- and a supersolution.
	\end{defn}

	To obtain the desired result, we need to impose some boundedness conditions on $\Theta$.		
	Let $\|\cdot\|$ denote both the Euclidean norm on $\mathbb{R}^d$ and the spectral norm on $\mathbb{S}$. For $k\in\mathfrak{L}$, define
	\begin{equation}\label{def:extended_norm_Levy}
		\|k\|:=\int_{\mathbb{R}^d}\left(\|z\|^2\wedge\|z\|\right)\,|k(dz)|,
	\end{equation}
	which is clearly an extended norm on $\mathfrak{L}$ since $k(\{0\})=0$ for all $k\in\mathfrak{L}$.
	
	Moreover, define (extended) norms for $\beta=(\beta_0,\ldots,\beta_d)\in(\mathbb{R}^d)^{d+1}$, $\alpha=(\alpha_0,\ldots,\alpha_d)\in\mathbb{S}^{d+1}$ and $\nu=(\nu_0,\ldots,\nu_d)\in\mathfrak{L}^{d+1}$ by 
	\begin{equation}\label{eq:ExtendedNorm1}
		\trinorm{\beta}:=\sup_{x\in\mathbb{R}^d} \frac{\|\beta_0+(\beta_1,\ldots,\beta_d)\,x\|}{\|x\|+1},
	\end{equation}
	\begin{equation}\label{eq:ExtendedNorm2}
		\trinorm{\alpha}:=\sup_{x\in\mathbb{R}^d} \frac{\|\alpha_0+(\alpha_1,\ldots,\alpha_d)\,x\|}{\|x\|+1},
	\end{equation}
	\begin{equation}\label{eq:ExtendedNorm3}
		\trinorm{\nu}:=\sup_{x\in\mathbb{R}^d} \frac{\|\nu_0+(\nu_1,\ldots,\nu_d)\,x\|}{\|x\|+1}.
	\end{equation}	
	By the properties of $\|\cdot\|$ it follows that each map in \eqref{eq:ExtendedNorm1} - \eqref{eq:ExtendedNorm3} satisfy the defining properties of an (extended) norm.
	Next, we state similar conditions as in \cite{neufeld_nonlinear_2016}, \cite{fadina_affine_2019}. 
	
	\begin{con}\label{con:C}
		Let $\Theta\subseteq(\mathbb{R}^d)^{d+1}\times\mathbb{S}^{d+1}\times\mathfrak{L}^{d+1}$ be non-empty, closed and such that the following holds. 
		For all $x\in\mathbb{R}^d$, we have $\mathcal{P}_x(\Theta)\neq\emptyset$ and
		\begin{equation}\label{eq:conC2}
			\quad\lim_{\delta\downarrow0} \mathcal{K}_{\delta}(x)=0,
		\end{equation}
		where for $\delta>0$,
		\begin{equation}
			\mathcal{K}_\delta(x):=\mathcal{K}_{\delta} (x;\Theta):=\sup_{(\beta,\alpha,\nu)\in\Theta}\, \int_{\|z\|\leq\delta}\|z\|^2\,\nu(x;dz).
		\end{equation}
		Further,
		\begin{equation}\label{eq:conC1}
			\mathcal{K}:=\mathcal{K}(\Theta):=\sup_{(\beta,\alpha,\nu)\in\Theta} \Big\{\trinorm{\beta}+\trinorm{\alpha}+\trinorm{\nu}\Big\}<\infty.
		\end{equation}
	\end{con}
	
	%The definitions in \eqref{eq:ExtendedNorm1} - \eqref{eq:ExtendedNorm3} strongly resemble the definition of an operator norm.
	%Only the denominator was chosen to be $\|x\|+1$ instead of $\|x\|$ in order to take the constant components of the affine characteristics into account.
	Note that \eqref{eq:conC1} is equivalent to $\Theta$ being bounded with respect to the operator norm induced by the Euclidean norm on $\mathbb{R}^{d+1}$.
	In particular, \eqref{eq:conC1} implies that the linear components of the parameters are uniformly bounded.
	\begin{cor}\label{cor:lin_bound}
		Let $\Theta\subseteq(\mathbb{R}^d)^{d+1}\times\mathbb{S}^{d+1}\times\mathfrak{L}^{d+1}$ satisfy Asssumption \ref{con:C}. Then
		\begin{equation}\label{ineq:lin_bound}
			\sup_{(\beta,\alpha,\nu)\in\Theta}\Big\{\|(\beta_1,\ldots,\beta_d)\|_{op}+\|(\alpha_1,\ldots,\alpha_d)\|_{op}+\|(\nu_1,\ldots,\nu_d)\|_{op}\Big\}\leq 3\,\mathcal{K},
		\end{equation}
		where $\beta=(\beta_0,\beta_1,\ldots,\beta_d)$, $\alpha=(\alpha_0,\alpha_1,\ldots,\alpha_d)$ and $\nu=(\nu_0,\nu_1,\ldots,\nu_d)$, and $\|\cdot\|_{op}$ denote the extended operator norms induced by $\|\cdot\|$.
	\end{cor}
	\begin{proof}
		By the definition of the operator norm and the triangle inequality, we have that
		\begin{align}\label{ineq:lin_bound_proof}
			\|(\beta_1,\ldots,\beta_d)\|_{op}
			& = \sup_{x\in\mathbb{R}^d;\,\|x\|=1} \left\|(\beta_1,\ldots,\beta_d)\,x\right\|\nonumber\\
			&\leq \sup_{x\in\mathbb{R}^d;\,\|x\|=1} \left\|\beta_0+(\beta_1,\ldots,\beta_d)\,x\right\|+\|\beta_0\|\nonumber\\
		    & = \sup_{x\in\mathbb{R}^d;\,\|x\|=1} 2\,\frac{\|\beta_0+(\beta_1,\ldots,\beta_d)\,x\|}{\|x\|+1}+\frac{\|\beta_0+(\beta_1,\ldots,\beta_d)\,0\|}{\|0\|+1} \nonumber\\
			&\leq 
			3\,\trinorm{\beta}.
		\end{align}			
		Analogously, we can show \eqref{ineq:lin_bound_proof} for $\alpha$ and $\nu$.
	\end{proof}

	In order to avoid subtleties related to measurability, from now on we work with the augmentation $\mathbb{F}^P_+$, as this does not affect the 
	the semimartingale characteristics, see Remark \ref{rem:filtration-closedness} Point \ref{rem:filtration}. 
	
	We start with some auxiliary results before we turn to the value function.
	We follow the line of argument in the proof of \cite[Lemma 3]{fadina_affine_2019} and adapt it to the case with jumps and consider a $d$-dimensional setting.
	\begin{lem}\label{lem:TimeCont0}
		Let $\Theta\subseteq(\mathbb{R}^d)^{d+1}\times\mathbb{S}^{d+1}\times\mathfrak{L}^{d+1}$ satisfy Assumption \ref{con:C}, and let $1\leq p\leq 2$.
		Then there exists an $\epsilon:=\epsilon(p)>0$ such that for all $0\leq t\leq \epsilon$, $x\in\mathbb{R}^d$ and $P\in\mathcal{P}_x(\Theta)$, we have
		\begin{equation}
			E_P \left[ \sup_{0\leq s\leq t} \|X_s-X_0\|^p \right]\leq C_{\mathcal{K},p}\,(1+\|x\|)^p\,(t^p+t^{\frac{p}{2}})
		\end{equation}
		for some constant $C_{\mathcal{K},p}$ independent of $t$, $x$ and $P$.
	\end{lem}
	\begin{proof}
		For some $x\in \mathbb{R}^d$, fix a $P\in\mathcal{P}_x(\Theta)$ and let $(b^P,a^P,k^P)$ denote the associated differential characteristics.
		Then we have 
		\begin{align}\label{ineq:cont0_drift}
			E_P\left[\sup_{0\leq s\leq t}\left\|B^P_s\right\|^p\,\right]
			&=
			E_P\left[\sup_{0\leq s\leq t}\left\|\int_0^s b^P_u\,du\right\|^p\,\right]\nonumber\\ 
			% differentiable characteristics
		%	&\leq 
		%	\expl{} \expl{E_P\left[\sup_{0\leq s\leq t}\left(\int_0^s \left\|b^P_u\right\|\,du\right)^p\,\right] } \nonumber\\ 
			% pulling in norm
			\expl{} & = 
			E_P\left[\left(\int_0^t \left\|b^P_u\right\|\,du\right)^p\,\right] \nonumber\\ 
			% supremum
			&\leq 
			\mathcal{K}^p\, E_P\left[\left(\int_0^t \left(\left\|X_u\right\|+1\right)\,du\right)^p\,\right] \nonumber\\ 
			% Con 2.1.a
		     & {\leq} 
		 {\mathcal{K}^p\,t^p\, E_P\left[\left(\sup_{0\leq s\leq t} \left\|X_s\right\|+1\right)^p\,\right]}\nonumber\\ 
			% supremum
			& \leq 
			\mathcal{K}^p\,t^p\, E_P\left[\left(\sup_{0\leq s\leq t} \left\|X_s-X_0\right\|+\|X_0\|+1\right)^p\,\right]\nonumber\\ 
			% triangle ineq
			&\leq 
			\mathcal{K}^p\,t^p\,2^{p-1} \left(E_P\left[\sup_{0\leq s\leq t} \left\|X_s-X_0\right\|^p\right]+\left(\|x\|+1\right)^p\right).
			% convexity of x->x^p
		\end{align}
		The last step follows by Jensen's inequality, as $x\rightarrow x^p$ is convex on $\mathbb{R}_+$ for $p\geq 1$. Moreover, we will use repeatedly the following inequality for a different number of non-negative summands, i.e., 
		\begin{equation}
			\left(\sum_{i=1}^n y_i\right)^p\leq n^{p-1}\,\sum_{i=1}^n y_i^p.
		\end{equation}
		By the Burkholder-Davis-Gundy inequality in \eqref{eq:BDG}, we have for the continuous local martingale part $\,^c\!M^{P}$ of $X$ that 
		\begin{align}\label{ineq:cont0_cMart}
			E_P\left[\sup_{0\leq s\leq t}\left\|\,^c\!M^{P}_s\right\|^p\right]
			&\leq 
			C_{p,d}\,E_P\left[\Big\|\left[\,^c\!M^{P}\right]_t\Big\|^{\frac{p}{2}}\right]\nonumber\\
			% BDG inequality
			&= 
			C_{p,d}\,E_P\left[\Big\|\int_{0}^t a^P_u\,du\,\Big\|^{\frac{p}{2}}\right]\nonumber\\
			% definition of quadratic variation
			%\expl{} & \expl{\leq} 
			%\expl{} \expl{C_{p,d}\,E_P\left[\left(\int_{0}^t \left\|a^P_u\right\|\,du\right)^{\frac{p}{2}}\,\right]}\nonumber\\
			% triangle inequality
			&\leq 
			{C_{p,d}\,\mathcal{K}^{\frac{p}{2}}\,E_P\left[\left(\int_{0}^t \left(\left\|X_u\right\|+1\right)\,du\right)^{\frac{p}{2}}\,\right]}\nonumber\\
			% boundedness condition
			 & {\leq} 
			 {C_{p,d}\,\mathcal{K}^{\frac{p}{2}}\,t^{\frac{p}{2}}\,E_P\left[\left(\sup_{0\leq s\leq t} \left\|X_s\right\|+1\right)^{\frac{p}{2}}\,\right]}\nonumber\\
			% supremum
		     &{\leq} 
			 {C_{p,d}\,\mathcal{K}^{\frac{p}{2}}\,t^{\frac{p}{2}}\,E_P\left[\left(\sup_{0\leq s\leq t} \left\|X_s\right\|+1\right)^{p}\,\right]}\nonumber\\
			% x < x^2 for all x>1
			&\leq 
			C_{p,d}\,\mathcal{K}^{\frac{p}{2}}\,t^{\frac{p}{2}}\,2^{p-1}\left(E_P\left[\sup_{0\leq s\leq t} \left\|X_s-X_0\right\|^p\right]+\left(\|x\|+1\right)^p\right),
			% convexity of x -> x^p
		\end{align}
		where $C_{p,d}$ is the constant from the Burkholder-Davis-Gundy inequality.
		Similarly, we obtain for the purely discontinuous local martingale part $\,^j\!M^{P}$,
		\begin{align}\label{ineq:cont0_dMart}
			E_P\left[\sup_{0\leq s\leq t}\left\|\,^j\!M^{P}_s\right\|^p\right]
			&\leq 
			C_{p,d}\,E_P\left[\Big\|\left[\,^j\!M^{P}\right]_t\Big\|^{\frac{p}{2}}\right]\nonumber\\
			% BDG inequality
			%& = 
			%C_{p,d}\,E_P\left[\Big\|\int_0^t \int_{\mathbb{R}^d}h(z)\,h(z)^T\,\mu(ds,dz)\Big\|^ {\frac{p}{2}}\right]\nonumber\\
			% definition of quadratic variation
			&\leq C_{p,d}\,E_P\left[\left(\int_{0}^t\int_{\mathbb{R}^d}\|h(z)\|^2\,\mu(ds,dz)\right)^{\frac{p}{2}}\, \right]\nonumber\\
			% triangle inequality
			& \leq C_{p,d}\,C_h^{p}\,E_P\left[\left(\int_{0}^t\int_{\mathbb{R}^d}(\|z\|^2\wedge\,1)\,\mu(ds,dz)\right)^{\frac{p}{2}}\, \right] \nonumber\\
			% bound for truncation function  
			& \leq C_{p,d}\,C_h^{p}\,E_P\left[\int_{0}^t\int_{\mathbb{R}^d}(\|z\|^2\wedge\,1)\,\mu(ds,dz) \right]^{\frac{p}{2}} \nonumber\\ 
			% Jensen's inequality x->x^p/2, p<=2
			 & \leq C_{p,d}\,C_h^{p}\,E_P\left[\int_{0}^t\int_{\mathbb{R}^d}(\|z\|^2\wedge\,1)\,k_u^P(dz)\,du \right]^{\frac{p}{2}} \nonumber\\
			% compensator martingale property
			&\leq C_{p,d}\,C_h^{p}\,E_P\left[\int_{0}^t\int_{\mathbb{R}^d}(\|z\|^2\wedge\,\|z\|)\,k_u^P(dz)\,du \right]^{\frac{p}{2}}\nonumber\\
			% |z|^2 \wedge 1  < |z|^2 \wedge |z|
			&\leq C_{p,d}\,C_h^{p}\,\mathcal{K}^{\frac{p}{2}}\,E_P\left[\int_{0}^t\left(\|X_u\|+1\right)\,du \right]^{\frac{p}{2}}\nonumber\\
			% boundedness condition
			&\leq C_{p,d}\,C_h^{p}\,\mathcal{K}^{\frac{p}{2}}\,t^{\frac{p}{2}}\, E_P\left[\sup_{0\leq s\leq t}\|X_s\|+1 \right]^{\frac{p}{2}}\nonumber\\
			% supremum
			&\leq C_{p,d}\,C_h^{p}\,\mathcal{K}^{\frac{p}{2}}\,t^{\frac{p}{2}}\, E_P\left[\sup_{0\leq s\leq t}\|X_s\|+1 \right]\nonumber\\
			% (1+epsilon)^p/2 <= 1+epsilon for 0<=p<=2
			&\leq C_{p,d}\,C_h^{p}\,\mathcal{K}^{\frac{p}{2}}\,t^{\frac{p}{2}}\, E_P\left[\sup_{0\leq s\leq t}\left(\|X_s\|+1\right)^p \right]\nonumber\\
			% 1+epsilon <= (1+epsilon)^p for p>=1
			&\leq C_{p,d}\,C_h^{p}\,\mathcal{K}^{\frac{p}{2}}\,t^{\frac{p}{2}}\,2^{p-1}\left(E_P\left[\sup_{0\leq s\leq t}\|X_s-X_0\|^p \right]+(\|x\|+1)^p\right),
			% convexity of x -> x^p and x < x^2	
		\end{align}
		where $C_h$ depends only on $h$ and is such that $\|z-h(z)\|\leq C_h\,(\|z\|^2\wedge\|z\|)$ and $\|h(z)\|\leq C_h\,(\|z\|\wedge 1)$ for all $z\in\mathbb{R}^d$, and $\|z\|,\|h(z)\|\leq C_h\,(\|z\|^2\wedge\|z\|)$ outside some neighbourhood of zero.\footnote{
			Such a constant exists.
			By the definition of $h$, there exists a $\delta \in ]0,1[$ such that $h(z)=z$ for all $z$ with $\|z\|\leq \delta$.
			Fix such a $\delta$ and set $C_h:=\delta^{-2}\,\|h\|_\infty$. 
			%Then $\|h(z)\|\leq C_h\,(\|z\|\wedge 1)$ and $\|z-h(z)\| \leq C_h\,(\|z\|^2 \wedge \|z\|)$ for all $z\in\mathbb{R}^d$.
			Then the desired inequalities hold.
			For a detailed proof, see Corollary \ref{cor:truncation}.
		}
		Further, for the unbounded jump part $J^{P}$,
		\begin{align}\label{ineq:cont0_jump}
			E_P\left[\sup_{0\leq s\leq t}\left\| J^{P}_s\right\|^p\right]
			&= 
			{E_P\left[\sup_{0\leq s\leq t}\left\| \int_0^s\int_{\mathbb{R}^d}[z-h(z)]\,\mu(du,dz)\right\|^p\,\right]}\nonumber\\
			% definition
		%	&\leq 
		%	\expl{} \expl{E_P\left[\sup_{0\leq s\leq t} \left(\int_0^s\int_{\mathbb{R}^d}\left\|z-h(z)\right\|\,\mu(du,dz)\right)^p\,\right]}\nonumber\\
			% triangle inequality
		 &{=} 
			E_P\left[ \left(\int_0^t\int_{\mathbb{R}^d}\left\|z-h(z)\right\|\,\mu(du,dz)\right)^p\,\right]\nonumber\\
			% supremum of integral with non-negative integrand
			&\leq 
			C_h^p\, E_P\left[ \left(\int_0^t\int_{\mathbb{R}^d}\left(\left\|z\right\|^2\wedge\|z\|\right)\,k^P_u(dz)\,du\right)^p\,\right]\nonumber\\
			% boundedness of truncation function
			&\leq 
			C_h^p\,\mathcal{K}^p\, E_P\left[ \left(\int_0^t\left(\|X_u\|+1\right)\,du\right)^p\,\right]\nonumber\\
			% boundedness condition
			%& {\leq} 
		    %{C_h^p\,\mathcal{K}^p\,t^p\, E_P\left[ \left(\sup_{0\leq s\leq t}\|X_s\|+1\right)^p\,\right]}\nonumber\\
			% supremum
			&\leq 
			C_h^p\,\mathcal{K}^p\,t^p\,2^{p-1} \left(E_P\left[ \sup_{0\leq s\leq t}\|X_s-X_0\|^p\right]+(\|x\|+1)^p\right).
			% convexity of x -> x^p
		\end{align}
		We can choose $C_{p,d},\,C_h\geq 1$.
		Then by the canonical semimartingale decomposition in \eqref{eq:CanonicalDecomposition}, the triangle inequality, and the inequalities \eqref{ineq:cont0_drift} - \eqref{ineq:cont0_jump}, we have for all $\epsilon\geq t$,
		\begin{align}
			E_P\left[\sup_{0\leq s\leq t}\|X_s-X_0\|^p\right]
			% &=E_P\left[\sup_{0\leq s\leq t}\|B^P_s+\,^c\!M^{P}_s+\,^j\!M^{P}_s+J^{P}_s\|^p\right]\nonumber\\
			% canonical decomposition
			% & {\leq}
			% {E_P\left[\sup_{0\leq s\leq t}\Big(\|B^P_s\|+\|\,^c\!M^{P}_s\|+\|\,^j\!M^{P}_s\|+\|J^{P}_s\|\Big)^p\right]}  \nonumber\\
			% triangle inequality
			&\leq 
			4^{p-1}\, E_P\left[\sup_{0\leq s\leq t}\|B^P_s\|^p+\sup_{0\leq s\leq t} \|\,^c\!M^{P}_s\|^p+\sup_{0\leq s\leq t}\|\,^j\!M^{P}_s\|^p+\sup_{0\leq s\leq t}\|J^{P}_s\|^p\right]\nonumber\\
			% convexity of x -> x^p
		     & {\leq} 
			 {4^{p-1}\,2^p\,C_h^p\,C_{p,d}\left(\,\mathcal{K}^p\,t^p +\mathcal{K}^{\frac{p}{2}}\,t^{\frac{p}{2}}\right)\!\! \left(E_P\left[ \sup_{0\leq s\leq t}\|X_s-X_0\|^p\right] +(\|x\|+1)^p\right) } \nonumber\\
			% inequalities above
			& \leq \underbrace{2^{3p-2}\,C_h^p\,C_{p,d}\,(\mathcal{K}^p+1)}_{=:\,\tilde{C}}\,(t^p+t^{\frac{p}{2}})\!\left(E_P\left[ \sup_{0\leq s\leq t}\|X_s-X_0\|^p\right]+(\|x\|+1)^p\right) \nonumber\\
			% choosing C_p,d >= 1 and moving constants around
			&\leq
			\tilde{C}\,(\epsilon^p+\epsilon^{\frac{p}{2}})\,E_P\left[ \sup_{0\leq s\leq t}\|X_s-X_0\|^p\right]+\tilde{C}\,(t^p+t^{\frac{p}{2}})(\|x\|+1)^p.
			% epsilon >= t	
		\end{align}
		Fix a sufficiently small $\epsilon:=\epsilon(p)>0$ such that $1-\tilde{C}\,(\epsilon^p+\epsilon^{\frac{p}{2}})>0$.
		Then we obtain
		\begin{equation}
			E_P\left[\sup_{0\leq s\leq t}\|X_s-X_0\|^p\right]\leq \underbrace{\frac{\tilde{C}}{1-\tilde{C}\,(\epsilon^p+\epsilon^{\frac{p}{2}})}}_{=:C_{\mathcal{K},p}}\,(\|x\|+1)^p\,(t^p+t^{\frac{p}{2}}),
		\end{equation}
		which proves Lemma \ref{lem:TimeCont0}.
	\end{proof}
	\begin{cor}\label{cor:TimeCont0}
		Let $\Theta\subseteq(\mathbb{R}^d)^{d+1}\times\mathbb{S}^{d+1}\times\mathfrak{L}^{d+1}$ satisfy Assumption \ref{con:C}, and let $1\leq p\leq 2$.
		Then there exists an $\epsilon:=\epsilon(p)>0$ such that for all $0\leq t\leq \epsilon$ and $x\in\mathbb{R}^d$, 
		\begin{equation}
			\mathcal{E}^x \left( \sup_{0\leq s\leq t} \|X_s-X_0\|^p \right)\leq C_{\mathcal{K},p}\,(1+\|x\|)^p\,(t^p+t^{\frac{p}{2}})
		\end{equation}
		for some constant $C_{\mathcal{K},p}$ independent of $t$ and $x$.
	\end{cor}
	\begin{proof}
		Follows immediately from Lemma \ref{lem:TimeCont0}.
	\end{proof}
	\begin{lem}\label{lem:integrability}
		Let $\Theta\subseteq(\mathbb{R}^d)^{d+1}\times\mathbb{S}^{d+1}\times\mathfrak{L}^{d+1}$ satisfy Assumption \ref{con:C}.
		Then for all $x\in\mathbb{R}^d$ and $t\geq0$,
		\begin{equation*}
			\mathcal{E}^x(\|X_t\|)<\infty.
		\end{equation*}
	\end{lem}
	\begin{proof}
		Let $\epsilon:=\epsilon(1)>0$ be the constant from Corollary \ref{cor:TimeCont0}, and let $x\in\mathbb{R}^d$.
		For $t\leq\epsilon$, the statement follows directly from Corollary \ref{cor:TimeCont0}.
		Suppose that $t>\epsilon$.
		The tower property in \eqref{eq:Aggregation} and the Markov property in \eqref{eq:markov} give
		\begin{align}
			\mathcal{E}^x(\|X_t\|)
			&
			= \mathcal{E}^x( \mathcal{E}^x_{t-\epsilon}(\|X_t\|)) \nonumber\\
			% tower property
			&=
			\mathcal{E}^x(\mathcal{E}^{X_{t-\epsilon}}(\|X_\epsilon\|))\nonumber\\
			% markov property
			&\leq
			\mathcal{E}^x(\mathcal{E}^{X_{t-\epsilon}}(\|X_\epsilon-X_0\|))+\mathcal{E}^x(\mathcal{E}^{X_{t-\epsilon}}(\|X_0\|))\nonumber\\
			% triangle inequality + sublinearity
			&=
			\mathcal{E}^x(\mathcal{E}^{X_{t-\epsilon}}(\|X_\epsilon-X_0\|))+\mathcal{E}^x(\mathcal{E}^{y}(\|X_0\|)\vert_{y=X_{t-\epsilon}})\nonumber\\
			&=
			\mathcal{E}^x(\mathcal{E}^{X_{t-\epsilon}}(\|X_\epsilon-X_0\|))+\mathcal{E}^x(\|y\|)\vert_{y=X_{t-\epsilon}}\nonumber\\
			&\leq 
			\mathcal{E}^x\left(\mathcal{E}^{X_{t-\epsilon}}\left(\sup_{0\leq u\leq \epsilon}\|X_u-X_0\|\right)\right)+\mathcal{E}^x(\|X_{t-\epsilon}\|)\nonumber\\
			% supremumum (left) + markov property (right)
			&\leq 
			C_{\mathcal{K},1}\,(\epsilon+\epsilon^{\frac{1}{2}})\,(\mathcal{E}^x(\|X_{t-\epsilon}\|)+1)+\mathcal{E}^x(\|X_{t-\epsilon}\|) \label{eq:Corrections_Explanation_1}\\
			% corollary above + markov property
			&\leq 
			\left(C_{\mathcal{K},1}\,(\epsilon+\epsilon^{\frac{1}{2}})+1\right)\,(\mathcal{E}^x(\|X_{t-\epsilon}\|)+1), \label{ineq:Cont_t-eps}
		\end{align}
		where we use Corollary \ref{cor:TimeCont0} in \eqref{eq:Corrections_Explanation_1}.
		For every $t>\epsilon$, there exists an $N\in\mathbb{N}$ such that $N\epsilon < t\leq (N+1)\,\epsilon$.
		Hence, the statement follows by repeating the procedure $N$ times and applying Corollary \ref{cor:TimeCont0} to $\mathcal{E}^x(\|X_{t-N\epsilon}\|)$.
	\end{proof}
	
	Now we can turn to the value function $v$. We start by proving its continuity and then show that it is a viscosity solution of \eqref{eq:PIDE}.
	
	\begin{lem}\label{lem:ValueF_ts}
		Let $\Theta\subseteq(\mathbb{R}^d)^{d+1}\times\mathbb{S}^{d+1}\times\mathfrak{L}^{d+1}$ satisfy Assumption \ref{con:C} and $\phi\in\textnormal{Lip}_b(\mathbb{R}^d)$.
		The value function $v$ from \eqref{eq:def-value-function} satisfies for all $t,s\geq 0$ and $x\in\mathbb{R}^d$,
		\begin{equation}
			v(t+s,x)=\mathcal{E}^x(v(t,X_s)).
		\end{equation}
	\end{lem}
	\begin{proof} The proof is essentially the same as in \cite[Lemma 5.1]{neufeld_nonlinear_2016}.
		Let $s,t\geq 0$ and $x\in\mathbb{R}^d$.
		By the definition of the value function and the Markov property from \eqref{eq:markov},
		\begin{align}
			v(t,X_s)=\mathcal{E}^{X_s}(\phi(X_t))=\mathcal{E}^x_s(\phi(X_{t+s})).
		\end{align}
		Applying $\mathcal{E}^x(\cdot)$ on both sides, the dynamic programming property in \eqref{eq:Aggregation} yields
		\begin{equation}
			\mathcal{E}^x(v(t,X_s))=\mathcal{E}^x(\mathcal{E}^x_s(\phi(X_{t+s}))=\mathcal{E}^x(\varphi(X_{t+s}))=v(t+s,x).
		\end{equation}
	\end{proof}
	
	\begin{lem}\label{lem:JointContinuity}
		Let $\Theta\subseteq(\mathbb{R}^d)^{d+1}\times\mathbb{S}^{d+1}\times\mathfrak{L}^{d+1}$ satisfy Assumption \ref{con:C} and $\phi\in\textnormal{Lip}_b(\mathbb{R}^d)$ with Lipschitz constant $L_\phi$.
		Let $v$ be the value function from \eqref{eq:def-value-function}.
		Then for fixed $x\in\mathbb{R}^d$, the map $t\mapsto v(t,x)$ is locally $\frac{1}{2}$-Hölder continuous.
		Moreover, if $\cup_{x \in K} \mathcal{P}_x(\Theta)$ is sequentially compact in itself (with respect to weak convergence of measures)\footnote{i.e. every sequence of measures in $\cup_{x \in K} \mathcal{P}_x(\Theta)$ has a convergent subsequence converging with respect to the weak convergence of measures to another measure in $\cup_{x \in K} \mathcal{P}_x(\Theta)$.} for compact $K \subset \mathbb{R}^d$, then for fixed $t \geq 0$, $x\mapsto v(t,x)$ is upper semicontinuous.  
	%	In particular, $x\mapsto v(t,x)$ is Lipschitz continuous with Lipschitz constant $L_\phi$ and $t\mapsto v(t,x)$ is locally $\frac{1}{2}$-Hölder continuous.
	\end{lem}
	\begin{proof} 
	%	For the Lipschitz continuity, observe that for all $t\geq 0$ and $x,y\in\mathbb{R}^d$,
	%	\begin{align}
	%		|v(t,x)-v(t,y)|&=\left|\mathcal{E}^x(\phi(X_t))-\mathcal{E}^y(\phi(X_t))\right|\nonumber\\
	%		&=|\mathcal{E}^x(\phi(X_t))-\mathcal{E}^x(\phi(X_t-x+y))|\nonumber\\
	%		&\leq \mathcal{E}^x\Big(\left|\phi(X_t)-\phi(X_t-x+y)\right|\Big)\nonumber\\
	%		&\leq L_\phi\,\|y-x\|.
	%	\end{align}
		For the local $\frac{1}{2}$-Hölder right-continuity, fix an $s\geq 0$ and let $t\geq0$. By the Markov property in \eqref{eq:markov}, we have for all $x\in\mathbb{R}^d$ that
		\begin{align}\label{eq:above}
			|v(s+t,x)-v(s,x)|
			&=\left|\mathcal{E}^x(\phi(X_{s+t}))-\mathcal{E}^x(\phi(X_s))\right|\nonumber\\
			&\leq\mathcal{E}^x\Big(|\phi(X_{s+t})-\phi(X_s)|\Big)\nonumber\\
			% sup f(x) = sup { f(x) - g(x) + g(x)} <= sup { f(x) - g(x)} + sup g(y)
			 & {=\mathcal{E}^x\Big(\mathcal{E}^{X_s}\Big(|\phi(X_t)-\phi(X_0)|\Big)\Big)} \nonumber\\
			% tower + markov property
			&\leq L_\phi\,\mathcal{E}^x\Big(\mathcal{E}^{X_s}\Big(\|X_t-X_0\|\Big)\Big)\nonumber\\
			% Lipschitz continuity of phi
			&\leq L_\phi\,\mathcal{E}^x\left(\mathcal{E}^{X_s}\left(\sup_{0\leq u\leq t}\|X_u-X_0\|\right)\right).
			% supremum
		\end{align}
		Let $\epsilon:=\epsilon(1)>0$ be the constant from Corollary \ref{cor:TimeCont0}.
		Then \eqref{eq:above} yields for all $0\leq t\leq\epsilon$,
		\begin{align}\label{ineq:HölderR}
			|v(s+t,x)-v(s,x)|&\leq L_\phi\,C_{\mathcal{K},1}\,(t+t^{\frac{1}{2}})\,\left(\mathcal{E}^x\left(\|X_s\|\right)+1\right),
		\end{align}
		and $\mathcal{E}^x(\|X_s\|)<\infty$ by Lemma \ref{lem:integrability}.
		Analogously, for the local $\frac{1}{2}$-Hölder left-continuity, we have that
		\begin{align}\label{ineq:HölderL}
			|v(s-t,x)-v(s,x)|&\leq L_\phi\,C_{\mathcal{K},1}\,(t+t^{\frac{1}{2}})\,\left(\mathcal{E}^x\left(\|X_{s-t}\|\right)+1\right)
		\end{align}
		for all $0\leq t\leq\epsilon$.
		Combining both \eqref{ineq:HölderR} and \eqref{ineq:HölderL} yields for all $-\epsilon\leq t\leq\epsilon$,
		\begin{align}
			|v(s+t,x)-v(s,x)|
			&\leq |v(s+t,x)-v(s,x)|+ |v(s-t,x)-v(s,x)| \nonumber \\
			% triangle inequality
			& \leq L_\phi\,C_{\mathcal{K},1}\,(|t|+|t|^{\frac{1}{2}})\,\left(\mathcal{E}^x\left(\|X_s\|\right)+\mathcal{E}^x\left(\|X_{s-t}\|\right)+2 \right)\nonumber\\
			% inequalities HölderL and HölderR
			& \leq L_\phi\,C_{\mathcal{K},1}\,(|t|+|t|^{\frac{1}{2}})\,\Big[\left( C_{\mathcal{K},1}\,(\epsilon+\epsilon^{\frac{1}{2}})+1\right)\,(\mathcal{E}^x(\|X_{s-\epsilon}\|)+1) \nonumber \\
			&\qquad\qquad\qquad\qquad +\left( C_{\mathcal{K},1}\,((\epsilon-t)+(\epsilon-t)^{\frac{1}{2}})+1\right)\,(\mathcal{E}^x(\|X_{s-\epsilon}\|)+1)+2 \Big] \label{eq:Corrections_Francesca_2}\\
			% inequality Cont_t-eps on both sublinear expectations
			&\leq L_\phi\,C_{\mathcal{K},1}\,(|t|+|t|^{\frac{1}{2}})\,\underbrace{(2\,C_{\mathcal{K},1}\,(\epsilon+\epsilon^{\frac{1}{2}})+1)\,(\mathcal{E}^x(\|X_{s-\epsilon}\|)+2)}_{=:H_{s,x}},
		\end{align}
		where we apply twice \eqref{ineq:Cont_t-eps} in \eqref{eq:Corrections_Francesca_2}.

The upper semicontinuity of $v$ follows directly by the assumptions on $\cup_{x \in K} \mathcal{P}_x(\Theta)$ and by Lemma 4.42 in \cite{hollender_levy-type_2016}. 		
%		For the joint continuity, fix $(t,x)\in\mathbb{R}_+\times\mathbb{R}^d$, and consider a sequence $(t_n,x_n)_{n\in\mathbb{N}}$ converging to $(t,x)$.
%		Then by the Lipschitz continuity and the local $\frac{1}{2}$-Hölder continuity,
%		\begin{align}
%			|v(t,x)-v(t_n,x_n)| &\leq|v(t,x)-v(t,x_n)|+|v(t,x_n)-v(t_n,x_n)|\nonumber\\
%			&\leq L_\phi\,\|x-x_n\|+L_\phi\,C_{\mathcal{K},1}\,H_{t,x}\,\left(|t-t_n|+|t-t_n|^{\frac{1}{2}}\right).
%		\end{align}
%		Finally, letting $n\rightarrow\infty$ yields the desired result.
	\end{proof}
	
%	Having established the continuity of the value function $v$, the existence follows by standard arguments of stochastic control.
\begin{rem}
    \begin{enumerate}
        \item Note that the value function $v$ in \eqref{eq:def-value-function} is not any longer continuous, as it is for example the case in \cite{neufeld_nonlinear_2016}, where the sublinear operator and the value function to the set $ \check{\mathcal{P}}_0(\Theta)$ in \eqref{eq:parameter-uncertainty} is considered. However, in our setting the equation in \eqref{eq:ContinuityFalse} is not any longer correct, which would allow to conclude the continuity of $v$. For $t \geq 0$ and $x,y \in \mathbb{R}^d$ it is not true that  
        \begin{align}
			\left|\mathcal{E}^x(\phi(X_t))-\mathcal{E}^y(\phi(X_t))\right|=|\mathcal{E}^x(\phi(X_t))-\mathcal{E}^x(\phi(X_t-x+y))|. \label{eq:ContinuityFalse}
		\end{align}
		\item The continuity of the value function $v$ has also been discussed in \cite{hollender_levy-type_2016}. In particular, in Lemma 4.42 in \cite{hollender_levy-type_2016} provides a criteria which guarantees the upper semicontinuity of $v$ and which we use in Lemma \ref{lem:JointContinuity}. However, as pointed out in \cite{hollender_levy-type_2016} such a criteria does not exist for the lower semicontinuity of $v$.
    \end{enumerate}
\end{rem}	
	\begin{prop}\label{prop:Existence}
		Let $\Theta\subseteq(\mathbb{R}^d)^{d+1}\times\mathbb{S}^{d+1}\times\mathfrak{L}^{d+1}$ satisfy Assumption \ref{con:C} and $\phi\in\textnormal{Lip}_b(\mathbb{R}^d)$.
		If $\cup_{x \in K} \mathcal{P}_x(\Theta)$ is sequentially compact in itself (with respect to weak convergence of measures) for compact $K \subset \mathbb{R}^d$, then the value function $v$ from \eqref{eq:def-value-function} is a viscosity subsolution of \eqref{eq:PIDE} in $\mathbb{R}_+\times\textnormal{S}$, and satisfies the initial condition \eqref{eq:PIDE-initial} on $\mathbb{R}^d$. If in addition the value function $v$ is also lower semicontinuous, then $v$ is a viscosity solution of \eqref{eq:PIDE} in $\mathbb{R}_+\times\textnormal{S}$, and satisfies the initial condition \eqref{eq:PIDE-initial} on $\mathbb{R}^d$. 
	\end{prop}
	
	\begin{proof}
		We largely follow the line of arguments presented in the proofs of \cite[Theorem 1]{fadina_affine_2019}, and \cite[Proposition 5.4]{neufeld_nonlinear_2016} for the jump part. The initial condition \eqref{eq:PIDE-initial} follows from the definition of the value function in \eqref{eq:def-value-function}.
		Further, the condition for the upper semicontinuity of the value function $v$ follows directly by Lemma \ref{lem:JointContinuity}.
		We prove that $v$ is a viscosity subsolution of the PIDE \eqref{eq:PIDE}.
		In case $v$ is also lower semicontinuous, the supersolution property follows analogously.
		
		Fix an interior point $(t,x)\in\mathbb{R}_+\times \textnormal{S}$, and let  $\psi\in\textnormal{C}^{2,3}_b(\mathbb{R}_+\times\mathbb{R}^d)$ with $\psi(t,x)=v(t,x)$ and $\psi\geq v$ on $\mathbb{R}_+\times\mathbb{R}^d$.
		By Lemma \ref{lem:ValueF_ts}, for all $0\leq s\leq t$, we have 
		\begin{equation}\label{ineq:0}
			0=\mathcal{E}^x\left(v(t-s,X_s)-v(t,x)\right)\leq \mathcal{E}^x\left(\psi(t-s,X_s)-\psi(t,x)\right).
		\end{equation}
		\begin{comment}
		First, consider the case $x\notin\textnormal{S}$.
		We have $\mathcal{E}^x(\psi(t-s,X_s))=\psi(t-s,x)$ due to Lemma \ref{lem:trivial_uncertainty_subset}, and $\mathcal{A}_{x}\psi(t,x)=0$ due to the definition of the operator $\mathcal{A}_y$ or, more precisely, due to the definition of the maps $\beta$, $\alpha$, $\nu$ in \eqref{eq:theta1} - \eqref{eq:theta3}.
		Hence, dividing \eqref{ineq:0} by $s>0$ and then letting $s\downarrow0$ yields the desired result, i.e.,
		\begin{equation}
		0\leq \frac{\mathcal{E}^x(\psi(t-s,X_s)-\psi(t,x))}{s}=\frac{\psi(t-s,x)-\psi(t,x)}{s}\overset{s\downarrow 0}{=}-\partial_t \psi(t,x)+\mathcal{A}_x\psi(t,x).
		\end{equation}
		
		Now consider the case $x\in\textnormal{S}$, which is much more delicate.
		\end{comment}
		
		In the following, we study the sublinear expectation on the right-hand side by considering the associated linear expectations.
		Therefore, fix a $P\in\mathcal{P}_x(\Theta)$, and let $(b^P,a^P,k^P)$ denote the differential characteristics of $X$ under $P$.
		For $0\leq s\leq t$,
		It\^{o}'s formula yields $P$-a.s. that
		\begin{flalign}
			\psi(&t-s,X_s)-\psi(t,x)\nonumber\\
			&=\int_0^s-\partial_t\psi(t-u,X_{u-})\,du+\int_0^s D_x\psi(t-u,X_{u-})^T\,b^P_u\,du\nonumber\\
			&\quad+\int_0^s D_x\psi(t-u,X_{u-})^T\,dM_u^{P}+\frac{1}{2}\int_0^s \tr\Big[D_{x}^2\psi(t-u,X_{u-})\,a^P_u\Big]\,du\label{eq:ItoVisc_1}\\
			&\quad+\int_0^s\int_{\mathbb{R}^d}\Big[\psi(t-u,X_{u-}+z)-\psi(t-u,X_{u-})
			-D_x\psi(t-u,X_{u-})^T\,h(z)\Big]\,\mu(du,dz), \label{eq:ItoVisc}
		\end{flalign}
		where $\mu$ is the jump measure, and $M^P=\,^c\!M^{P}+\,^j\!M^{P}$ is the local martingale part in the canonical decomposition of $X$ under $P$, cf. \eqref{eq:CanonicalDecomposition}.
		
		Since $\psi\in\textnormal{C}^{2,3}_b(\mathbb{R}_+\times\mathbb{R}^d)$, the gradient $D_x\psi$ is bounded.
		Thus, the integral with respect to $M^{P}$ is a local martingale starting at zero, cf. \cite[Corollary I.4.55]{jacod_limit_2003}.
		Let $\epsilon:=\epsilon(1)$ be the constant from Lemma \ref{lem:TimeCont0}. Then we obtain for all $0\leq s\leq \epsilon$,
		\begin{flalign}\label{ineq:3}
			E_P&\Bigg[\left[\int_{0}^{\cdot}D_x\psi(t-u,X_{u-})^T\,dM^P_u\right]_s\Bigg]\nonumber\\
			&=E_P\Bigg[\int_{0}^{s}D_x\psi(t-u,X_{u-})^T\,d[M^P]_u\,D_x\psi(t-u,X_{u-})\Bigg]\nonumber\\
			% definition quadratic variation
			%\expl{} & \expl{=} 
			%\expl{} \expl{E_P\Bigg[\int_{0}^{s}D_x\psi(t-u,X_{u-})^T\, \left(d[\,^c\!M^P]_u+d[\,^j\!M^P]_u\right)\,D_x\psi(t-u,X_{u-})\Bigg]}\nonumber\\
			% splitting into cont. and jump part
			&=
			E_P\Bigg[\int_{0}^{s}D_x\psi(t-u,X_{u-})^T\,a_u^P\;D_x\psi(t-u,X_{u-})\,du\Bigg]\nonumber\\
			% definition cont. quadratic variation
			&\quad
			+E_P\Bigg[\int_{0}^{s}D_x\psi(t-u,X_{u-})^T\,\int_{\mathbb{R}^d}h(z)\,h(z)^T\,d\mu(du,dz)\,D_x\psi(t-u,X_{u-})\Bigg]\nonumber\\
			% definition jump quadratic variation
			&\leq
			\|D_x\psi\|_\infty^2\,\left(E_P\left[\int_0^s\left\|a_u^P\right\|\,du\right]+ E_P\left[\int_0^s\int_{\mathbb{R}^d}\|h(z)\|^2\,\mu(du,dz)\right]\right)<\infty,
			% Lipschitz continuity of nabla phi and inequalities above
		\end{flalign}
		due to the inequalities \eqref{ineq:cont0_cMart} and \eqref{ineq:cont0_dMart}.
		Hence, the integral with respect to $M^P$ in \eqref{eq:ItoVisc_1} is a true martingale on $[0,\epsilon]$, cf. \cite[Corollary II.6.3]{protter_stochastic_2004}, and 
	    \begin{equation}
	        E_P \left[ \int_0^s D_x\psi(t-u,X_{u-})^T\,dM_u^{P} \right]=0.
	    \end{equation}
		To estimate the expectation of the other terms,
		note that $\psi$, $\partial_t\psi$, $D_x \psi$ and $D^2_{x}\psi$ are Lipschitz continuous in $t,x$ since $\psi\in\textnormal{C}_b^{2,3}(\mathbb{R}_+\times\mathbb{R}^d)$.
 		From now on we use the notation $L:=\max(L,L_t,L_x,L_{xx})<\infty$. 
		Observe that for the first term on the right-hand side of \eqref{eq:ItoVisc},
		\begin{equation}\label{eq:time}
			\int_0^s-\partial_t\psi(t-u,X_{u-})\,du=\int_0^s\Big(\partial_t\psi(t,x)-\partial_t\psi(t-u,X_{u-}) \Big)\,du-\int_0^s\partial_t\psi(t,x)\,du.
		\end{equation}
		For sufficiently small $s$, Lemma \ref{lem:TimeCont0} gives the estimation
		\begin{flalign}\label{ineq:1}
			 \Bigg \vert E_P\left[\int_0^s\Big(\partial_t\psi(t,x)-\partial_t\psi(t-u,X_{u-})\Big)\,du\right] \Bigg \vert 	
			%\expl{} &\expl{\leq} 
			%\expl{} \expl{E_P\left[\int_0^s\left|\partial_t\psi(t,x) -\partial_t\psi(t-u,X_{u-})\right|\,du\right]}\nonumber\\ 
			% absolute value
			&\leq L\, E_P\left[\int_0^s\left( u+\|X_{u-}-x\| \right)\,du\right]\nonumber\\ 
			% Lipschitz continuity
			&\leq L\,s\,E_P\left[\frac{s}{2}+\sup_{0\leq u\leq s}\|X_u-x\|\right]\nonumber\\ 
			% supremum
			&\leq L\,s\,\left(\frac{s}{2}+C_{\mathcal{K},1}(\|x\|+1)(s+s^{\frac{1}{2}})\right) \nonumber\\
			% Corollary
			%&\leq \expl{L\,\left(\frac{1}{2}+C_{\mathcal{K},1}\,(\|x\|+1)\right)(s^2+s^{\frac{3}{2}})} \nonumber\\
			% s/2+ C(|x|+1)(s+s^1/2)< (s+s^1/2)/2 + C(|x|+1)(s+s^1/2)=(1/2+C(|x|+1)(s+s^1/2)
			&{=} L\,\left(\frac{1}{2}+C_{\mathcal{K},1}\,(\|x\|+1)\right)(s^{\frac{1}{2}}+1)\,s^{\frac{3}{2}}.
		\end{flalign}
		Similarly for the second term on the right-hand side of \eqref{eq:ItoVisc}, we have
		\begin{align}
			\int_0^s D_x\psi(t-u,X_{u-})^T\,b^P_u\,du&=\int_0^s\Big( D_x\psi(t-u,X_{u-})-D_x\psi(t,x)\Big)^T\,b^P_u\,du\nonumber\\
			&\quad+\int_0^sD_x\psi(t,x)^T\,b^P_u\,du.
		\end{align}
		Applying Lemma \ref{lem:TimeCont0},  this yields for sufficiently small $0\leq s$ that
		\begin{align}\label{ineq:2}
			\Bigg \vert E_P&\Bigg[\int_0^s\Big( D_x\psi(t-u,X_{u-})-D_x\psi(t,x)\Big)^T\,b^P_u\,du\Bigg]\Bigg \vert  \nonumber\\
			% &\leq  E_P\Bigg[\int_0^s\left|\Big(D_x\psi(t-u,X_{u-})-D_x\psi(t,x)\Big)^T\,b^P_u\right|\,du\Bigg]\nonumber\\ 
			% absolute value
			&\leq L\, E_P\left[\int_0^s\left(u+\|X_{u-}-x\|\right)\,\|b^P_u\|\,du\right]\nonumber\\
			% Cauchy + Lipschitz continuity of D_x\psi + triangle inequality
			&\leq L\,\mathcal{K}\,E_P\left[\int_0^s \left(u+\|X_{u-}-x\|\right)\,(\|X_u\|+1)\,du\right] \nonumber\\
			&\leq L\,\mathcal{K}\,s\, E_P\left[\left(\frac{s}{2}+\sup_{0\leq u\leq s}\|X_{u}-x\|\right)\,\left(\sup_{0\leq u\leq s}\|X_u-x\|+\|x\|+1\right)\right]\nonumber\\
			&= L\,\mathcal{K}\,s \Bigg(\left(\frac{s}{2}+\|x\|+1\right)\,E_P\left[\sup_{0\leq u\leq s}\|X_{u}-x\|\right]+\frac{s}{2}\,(\|x\|+1)+E_P\left[\sup_{0\leq u\leq s}\|X_u-x\|^2\right] \Bigg)\nonumber\\
			% (sup|X_0-x|)^2=sup|X_0-x|^2 since x->x² is monotone on positive axis
			&\leq L\,\mathcal{K}\,s\Bigg(\left(\frac{s}{2}+\|x\|+1\right)\,C_{\mathcal{K},1}\,(s+s^{\frac{1}{2}})\,(\|x\|+1)+\frac{s}{2}\,(\|x\|+1)+C_{\mathcal{K},2}\,(s^2+s)\,(\|x\|+1)^2\Bigg)\nonumber\\
			% Corollary with p=1 and p=2
			%& = L\,\mathcal{K}\,s\left( (\|x\|+1)^2\,C_{\mathcal{K},1}\,(s+s^{\frac{1}{2}})+\frac{1}{2}\,(\|x\|+1)\,C_{\mathcal{K},1}\,(s^2+s^{\frac{3}{2}})+\frac{1}{2}\,(\|x\|+1)\,s+(\|x\|+1)^2\,C_{\mathcal{K},2}\,(s^2+s) \right) \nonumber\\
			&\leq %L\,\mathcal{K}\,(\|x\|+1)^2 \,\left(C_{\mathcal{K},1}+1+C_{\mathcal{K},2}\right)\,(s^3+s^{\frac{5}{2}}+s^2+s^{\frac{3}{2}})\nonumber\\
			% (|x|+1)/2 < (|x|+1)^2
			%&= 
			L\,\mathcal{K}\,(\|x\|+1)^2 \,\left(C_{\mathcal{K},1}+1+C_{\mathcal{K},2}\right)\,(s^{\frac{3}{2}}+s+s^{\frac{1}{2}}+1)\,s^{\frac{3}{2}}.
		\end{align}
		For the forth term on the right-hand side of \eqref{eq:ItoVisc},
		\begin{align}
			\int_0^s \tr\Big[D_{x}^2\psi(t-u,X_{u-})\,a^P_u\Big]\,du&=\int_0^s \tr\Big[\Big(D_{x}^2\psi(t-u,X_{u-})-D_{x}^2\psi(t,x)\Big)\,a^P_u\Big]\,du\nonumber\\
			&\quad + \int_0^s \tr\Big[D_{x}^2\psi(t,x)\,a^P_u\Big]\,du.
		\end{align}
		Since the trace is an inner product on $\mathbb{S}$, we get by the Cauchy-Schwarz inequality, we obtain that
		\begin{flalign}\label{ineq:4}
			E_P&\left[\int_0^s \tr\left[\Big(D_{x}^2\psi(t-u,X_{u-})-D_{x}^2\psi(t,x)\Big)\,a^P_u\right]\,du\right]\nonumber\\
			&\leq E_P\left[\int_0^s \sqrt{ \tr\left[\Big(D_{x}^2\psi(t-u,X_{u-})-D_{x}^2\psi(t,x)\Big)^2\right] \tr \left[\left(a_u^P\right)^2\right]}\,du\right] \nonumber\\ 
			% Cauchy-Schwarz
			&\leq E_P\left[\int_0^s \sqrt{d^2\, \left\|D_{x}^2\psi(t-u,X_{u-})-D_{x}^2\psi(t,x)\right\|^2\,\left\|a_u^P\right\| ^2}\,du\right] \nonumber\\ 
			% maximal singular value
			& = E_P\left[\int_0^s d \left\|D_{x}^2\psi(t-u,X_{u-})-D_{x}^2\psi(t,x)\right\|\,\left\|a_u^P\right\| \,du\right] \nonumber\\
			% solve root
			&\leq {d\,L\, E_P\left[\int_0^s \left(u+\|X_{u-}-x\|\right)\,\left\|a_u^P\right\|\,du\right]} \nonumber\\ 
			% Lipschitz continuity
		%	\expl{} & \leq \expl{d\,L\,\mathcal{K}\, E_P\left[\int_0^s\left(u+\|X_{u-}-x\|\right)\,(\left\|X_u\right\|+1)\,du\right]} \nonumber\\
		%	\expl{} & \leq \expl{d\,L\,\mathcal{K}\,s\, E_P\left[\left(\frac{s}{2}+\sup_{0\leq u\leq s}\|X_{u}-x\|\right)\left(\sup_{0\leq u\leq s}\left\|X_u-x\right\|+\|x\|+1\right)\right]} \nonumber\\ 
			% supremum
			&\leq d\,L\,\mathcal{K}\,(\|x\|+1)^2 \,\left(C_{\mathcal{K},1}+1+C_{\mathcal{K},2}\right)\,(s^{\frac{3}{2}}+s+s^{\frac{1}{2}}+1)\,s^{\frac{3}{2}},
		\end{flalign}
		where the last steps follows as in \eqref{ineq:2}.
		
		For the last term in \eqref{eq:ItoVisc}, i.e., the integral with respect to the jump measure $\mu$, we consider large and small jumps separately.
		Therefore, take some $0<\delta\leq1$ such that $h(z)=z$ for all $z$ with $\|z\|\leq\delta$.
		By a Taylor expansion and the Mean-Value Theorem, there exists for each $z$ a $\zeta_z\in\mathbb{R}^d$ such that $P$-a.s.,
		\begin{align}\label{eq:jumps}
			\int_0^s&\int_{\mathbb{R}^d}\Big[\psi(t-u,X_{u-}+z)-\psi(t-u,X_{u-})-D_x\psi(t-u,X_{u-})^T\,h(z)\Big]\,\mu(du,dz)\nonumber\\
			&=\int_0^s\int_{\|z\|\geq \delta}\Big[\psi(t-u,X_{u-}+z)-\psi(t-u,X_{u-})-D_x\psi(t-u,X_{u-})^T\,h(z)\Big]\,\mu(du,dz)\nonumber\\
			&\quad+\int_0^s\int_{\|z\|< \delta}\left[\psi(t,x+z)-\psi(t,x)-D_{x}\psi(t,x)^T\,z\right]\,\mu(du,dz)\nonumber\\
			&=\int_0^s\int_{\|z\|\geq \delta}\Big[\psi(t-u,X_{u-}+z)-\psi(t-u,X_{u-})-D_x\psi(t-u,X_{u-})^T\,h(z)\Big]\,\mu(du,dz)\nonumber\\
			&\quad+\frac{1}{2}\int_0^s\int_{\|z\|< \delta}\tr\Big[D_{x}^2\psi(t-u,X_{u-}+\zeta_z)\,z\,z^T\Big]\,\mu(du,dz).
		\end{align}
		Both terms are $P$-integrable due to Assumption \ref{con:C}.
		Moreover for the large jumps, 
		\begin{align}
			\int_0^s&\int_{\|z\|\geq \delta}\Big[\psi(t-u,X_{u-}+z)-\psi(t-u,X_{u-})-D_x\psi(t-u,X_{u-})^T\,h(z)\Big]\,\mu(du,dz)\nonumber\\
			&=\int_0^s \int_{\|z\|\geq \delta}\bigg[\psi(t-u,X_{u-}+z)-\psi(t,x+z)+\psi(t,x)-\psi(t-u,X_{u-})\nonumber\\
			&\qquad\qquad\qquad\qquad\qquad\qquad+\left.\Big(D_x\psi(t,x)-D_x\psi(t-u,X_{u-})\Big)^T\, h(z)\right]\,\mu(du,dz)\nonumber\\
			% part to estimate 
			&\quad+\int_0^s\int_{\|z\|\geq \delta}\Big[\psi(t,x+z)-\psi(t,x)-D_x\psi(t,x)^T\, h(z)\Big]\,\mu(du,dz).
		\end{align}
        And by the same arguments as above, we obtain for sufficiently small $s$ that
		\begin{align}\label{ineq:5bigjumps}
			\Bigg|E_P&\Bigg[\int_0^s \int_{\|z\|\geq \delta}\bigg[\psi(t-u,X_{u-}+z)-\psi(t,x+z)+\psi(t,x)-\psi(t-u,X_{u-})\nonumber\\
			&\qquad\qquad\qquad\quad\qquad\qquad\qquad\quad+\Big(D_x\psi(t,x)-D_x\psi(t-u,X_{u-})\Big)^T\, h(z)\bigg]\,\mu(du,dz)\Bigg]\Bigg|\nonumber\\
			% initial line
			% &\leq E_P\Bigg[\int_0^s \int_{\|z\|\geq \delta}\bigg[ \left|\psi(t-u,X_{u-}+z)-\psi(t,x+z)\right|+\left|\psi(t,x)-\psi(t-u,X_{u-})\right| \nonumber\\
			%&\qquad\qquad\qquad\quad\qquad\qquad\qquad\quad +\Big\|D_x\psi(t,x)-D_x\psi(t-u,X_{u-})\Big\|\, \|h(z)\|\bigg]\,\mu(du,dz)\Bigg] \nonumber\\
			&\leq E_P\left[ \int_0^s \int_{\|z\|\geq \delta}\Big[2\,L\,(u+\|X_{u-}-x\|)+L\,\|h(z)\|\,(u+\|X_{u-}-x\|)\Big]\,\mu(du,dz)\right]\nonumber\\
			% Lipschitz continuities
			&\leq L\,\left(2+C_h\right)\, E_P\left[\int_0^s(u+\|X_{u-}-x\|)\int_{\|z\|\geq \delta}\|z\|^2\wedge\|z\|\,k^P_u(dz)\,du\right]\nonumber\\
			% pulling factors and integral out, changing jump measure
			%& = L\,\left(2+C_h\right)\, E_P\left[\int_0^s(u+\|X_{u-}-x\|)\,\|k^P_u\|\,du\right] \nonumber\\
			%&\leq L\,\left(2+C_h\right)\,\mathcal{K}\,E_P\left[\int_0^s(u+\|X_{u-}-x\|)(\|X_u\|+1)\,du\right] \nonumber\\
			% condition and supremum
			%&\leq L\,\left(2+C_h\right)\,\mathcal{K}\,s\,E_P\left[\left(s+\sup_{0\leq u\leq s}\|X_u-x\|\right)\left(\sup_{0\leq u\leq s}\|X_u-x\|+\|x\|+1\right)\right] \nonumber\\
			% supremum + triangle inequality
			&\leq L\,\left(2+C_h\right)\,\mathcal{K}\,(\|x\|+1)^2 \,\left(C_{\mathcal{K},1}+1+C_{\mathcal{K},2}\right)\,(s^{\frac{3}{2}}+s+s^{\frac{1}{2}}+1)\,s^{\frac{3}{2}}.
		\end{align}
        For the expectation of the small jumps in \eqref{eq:jumps}, we have
		\begin{flalign}
			\Bigg|E_P&\left[\int_0^s\int_{\|z\|< \delta}\left[\psi(t,x+z)-\psi(t,x)-D_{x}\psi(t,x)^T\,z\right]\,\mu(du,dz)\right]\Bigg| \nonumber\\
			&=\frac{1}{2}\,\Bigg|E_P\left[\int_0^s\int_{\|z\|< \delta}\tr\Big[D_{x}^2\psi(t-u,X_{u-}+\zeta_z)\,z\,z^T\Big]\,\mu(du,dz)\right]\Bigg|\label{eq:TaylorExpansion}\\
			%&\leq \frac{1}{2}\, E_P\left[\int_0^s \int_{\|z\|< 	\delta}\sqrt{\tr\left[\Big(D_x^2\psi(t-u,X_{u-}+\zeta_z)\Big)^2\right]\,\tr\Big[\left(z\,z^T\right)^2\Big]}\,\mu(du,dz)\right] \nonumber\\
			% Cauchy-Schwarz
			%&\leq \frac{1}{2}\, E_P\left[\int_0^s \int_{\|z\|< \delta}\sqrt{d^2\,\Big\|D_{x}^2\psi(t-u,X_{u-}+\zeta_z)\Big\|^2\,\|z\|^4}\;\mu(du,dz) \right] \nonumber\\
			% singular values *d
			& = \frac{1}{2}\, E_P\left[\int_0^s \int_{\|z\|< \delta} d\,\Big\|D_{x}^2\psi(t-u,X_{u-}+\zeta_z)\Big\|\,\|z\|^2\;\mu(du,dz) \right] \nonumber\\
			% solve root
			&\leq\frac{d}{2}\,\|D_{x}^2\psi\|_\infty\, E_P\left[\int_0^s\int_{\|z\|< \delta}\|z\|^2\,k^P_u(dz)\,du\right]\nonumber\\
			% Hessian-matrix bounded + change of jump measure
			& \leq \frac{d}{2}\,\|D_{x}^2\psi\|_\infty\, E_P\left[\int_0^s \sup_{(\beta,\alpha,\nu)\in\Theta}\int_{\|z\|< \delta}\|z\|^2\,\nu(X_u;dz)\,du\right] \nonumber\\
			&=\frac{d}{2}\,\|D_{x}^2\psi\|_\infty\, E_P\left[\int_0^s\mathcal{K}_\delta(X_u)\,du\right]\nonumber\\
			% condition
			&\leq\underbrace{\frac{d}{2}\,\|D_{x}^2\psi\|_\infty\,E_P\left[\sup_{0\leq u\leq s}\mathcal{K}_\delta(X_u)\right]}_{=:C_{\delta,s}}\,s, \label{eq:DefinitionCDeltaS}
			% supremum
		\end{flalign}
		where such a $\zeta_z \in \mathbb{R}^d$ in \eqref{eq:TaylorExpansion} exists by using a Taylor expansion and $K_\delta(x)$ is as defined in Assumption \ref{con:C}.
		
		Choose $\hat{C}>0$ such that $\hat{C}\,s^{\frac{3}{2}}$ is equal or greater than the sum of the right-hand side of equations \eqref{ineq:1}, \eqref{ineq:2}, \eqref{ineq:4} and \eqref{ineq:5bigjumps}.
		For instance, consider
		\begin{equation}
		    \hat{C}:=L\,\left(4+d+C_h\right)\,(\mathcal{K}+1)\,(\|x\|+1)^2 \,\left(C_{\mathcal{K},1}+1+C_{\mathcal{K},2}\right)\,(s^{\frac{3}{2}}+s+s^{\frac{1}{2}}+1).
		\end{equation}
		Then, combining \eqref{eq:ItoVisc} with \eqref{eq:time} - \eqref{eq:DefinitionCDeltaS}, yields
		\begin{flalign}\label{ineq:combined}
			E_P&\left[\psi(t-s,X_s)-\psi(t,x)\right]\nonumber\\
			&\leq\hat{C}\,s^{\frac{3}{2}}+C_{\delta,s}\,s-\int_{0}^s\partial_t\psi(t,x)\,du\nonumber\\
			&\quad +E_P\Bigg[\int_0^sD_x\psi(t,x)^T\,b_u^P\,du+\int_0^s\tr\Big[D_{x}^2\psi(t,x)\,a^P_u\Big]\,du\nonumber\\
			&\qquad\qquad\qquad\qquad\qquad\qquad+\int_0^s\int_{\|z\|\geq \delta}\left[\psi(t,x+z)-\psi(t,x)-D_{x}\psi(t,x)^T\,h(z)\right]\,\mu(du,dz)\Bigg]\nonumber\\
			&=\hat{C}\,s^{\frac{3}{2}}+\left(C_{\delta,s}-\partial_t\psi(t,x)\right)s\nonumber\\
			&\quad+E_P\Bigg[\int_0^s \bigg \lbrace D_x\psi(t,x)^T\,b_u^P+\tr\Big[D_{x}^2\psi(t,x)\,a^P_u\Big]\nonumber\\
			&\qquad\qquad\qquad\qquad\qquad\qquad+\int_{\mathbb{R}^d}\left[\psi(t,x+z)-\psi(t,x)-D_{x}\psi(t,x)^T\,h(z)\right]\,k^P_u(dz)\bigg \rbrace\,du\Bigg]\nonumber\\
			&\quad - \underbrace{E_P\left[\int_0^s\int_{\|z\|< \delta}\left[\psi(t,x+z)-\psi(t,x)-D_{x}\psi(t,x)^T\,h(z)\right]\,\mu(du,dz)\right]}_{\leq C_{\delta,s}\,s\quad \text{due to \eqref{eq:DefinitionCDeltaS}}} \nonumber\\
			% due to inequality for small jumps
			&\leq \hat{C}\,s^{\frac{3}{2}}+\left(2\,C_{\delta,s}-\partial_t\psi(t,x)\right)s\nonumber\\
			&\quad+E_P\Bigg[\int_0^s\sup_{(\beta,\alpha,\nu)\in\Theta} \bigg\{D_x\psi(t,x)^T\,\beta(X_u)+\tr\Big[D_{x}^2\psi(t,x)\,\alpha(X_u)\Big]\nonumber\\
			&\qquad\qquad\qquad\qquad\qquad\qquad+\int_{\mathbb{R}^d}\left[\psi(t,x+z)-\psi(t,x)-D_{x}\psi(t,x)^T\,h(z)\right]\,\nu(X_u; dz)\bigg\}\,du\Bigg]\nonumber\\
			&=\hat{C}\,s^{\frac{3}{2}}+\left(2\,C_{\delta,s}-\partial_t\psi(t,x)\right)s+E_P\left[\int_0^s\mathcal{A}_{X_u}\psi(t,x)\,du\right]\nonumber\\
			&\leq\hat{C}\,s^{\frac{3}{2}}+\left(2\,C_{\delta,s}-\partial_t\psi(t,x)\right)s+s\,E_P\left[\sup_{0\leq u\leq s}\mathcal{A}_{X_u}\psi(t,x)\right],
		\end{flalign}
		where $\mathcal{A}_{y}$ is introduced in \eqref{eq:def-nonlinear-A}.
		Since $P\in\mathcal{P}_x(\Theta)$ was arbitrary, we conclude from \eqref{ineq:combined} and \eqref{ineq:0}, divided by $s>0$, that
		\begin{equation}
			0\leq \hat{C}\,s^{\frac{1}{2}}+2\,C_{\delta,s}-\partial_t\psi(t,x)+\mathcal{E}^x\left(\sup_{0\leq u\leq s}\mathcal{A}_{X_u}\psi(t,x)\right).
		\end{equation}
		
		\noindent We proceed by studying the convergence of the last term for $s\downarrow0$.
		Therefore, note that
		\begin{align}\label{ineq:A0}
			\Big| \mathcal{E}^x\left(\sup_{0\leq u\leq s}\mathcal{A}_{X_u}\psi(t,x)\right)-\mathcal{A}_x\psi(t,x)\Big|
			&\leq \mathcal{E}^x\left(\left| \sup_{0\leq u\leq s}\mathcal{A}_{X_u}\psi(t,x)-\mathcal{A}_x\psi(t,x)\right|\right) \nonumber\\
			&\leq \mathcal{E}^x\left(  \sup_{0\leq u\leq s} \left|\mathcal{A}_{X_u}\psi(t,x)-\mathcal{A}_x\psi(t,x)\right|\right).
		\end{align}
		
		Further, observe that we have $\mathcal{A}_{X_u}\psi(t,x)=0$ on $\{X_u\notin\textnormal{S}\}$.			
		On $\{X_u\in\textnormal{S}\}$, Corollary \ref{cor:lin_bound} together with the definition of $\mathcal{A}_y$ yields
		\begin{align}
			&| \mathcal{A}_{X_u}\psi(t,x) - \mathcal{A}_{x}\psi(t,x) |\nonumber\\
			%&\leq \Bigg| \sup_{(\beta,\alpha,\nu)\in\Theta}\Bigg\{D_{x}\psi(t,x)^T\, (\beta_1,\ldots,\beta_d)\,(X_u-x)+\frac{1}{2}\,\tr\Big[D_{x}^2\psi(t,x)\,(\alpha_1,\ldots,\alpha_d)\,(X_u-x)\Big] \nonumber\\
			%\expl{} &\quad\qquad\qquad\qquad \expl{+\int_{\mathbb{R}^d}\left[\psi(t,x+z)-\psi(t,x)-D_x\psi(t,x)^T\,h(z)\right]\,(\nu_1(dz),\ldots,\nu_d(dz))\,(X_u-x)\Bigg\} \Bigg|} \nonumber\\
			% supremum inequality
			&\leq \sup_{(\beta,\alpha,\nu)\in\Theta} \Bigg\{ \Big| D_{x}\psi(t,x)^T\, (\beta_1,\ldots,\beta_d)\,(X_u-x)\Big| +\frac{1}{2}\,\left|\tr\Big[D_{x}^2\psi(t,x)\,(\alpha_1,\ldots,\alpha_d)\,(X_u-x)\Big]\right| \nonumber\\
			&\quad\qquad\qquad\qquad+ \left| \int_{\mathbb{R}^d} \left[\psi(t,x+z)-\psi(t,x)- D_x\psi(t,x)^T\,h(z)\right] \,(\nu_1(dz),\ldots,\nu_d(dz))\,(X_u-x) \right|\Bigg\} \nonumber \\
			% triangle inequality
			&\leq3\,\mathcal{K}\,\|D_x\psi\|_\infty\,\|X_u-x\|+3\,\mathcal{K}\,\frac{d}{2}\,\|D_{x}^2\psi\|_\infty\,\|X_u-x\| \nonumber\\
			% inequality for linear part for \beta and \alpha
			&\quad+\int_{\|z\|\geq \delta} \Big[ \left|\psi(t,x+z)-\psi(t,x)\right| + \|D_x\psi(t,x)\|\,\|h(z)\| \Big] \,\left|(\nu_1(dz),\ldots,\nu_d(dz))\,(X_u-x)\right| \nonumber\\
			&\quad+\left| \int_{\|z\|< \delta} \left[\psi(t,x+z)-\psi(t,x)- D_x\psi(t,x)^T\,z\right] \,(\nu_1(dz),\ldots,\nu_d(dz))\,(X_u-x) \right|\nonumber\\
			% separate small and big jumps + take absolute value
			&\leq3\,\mathcal{K}\, \Big(\|D_x\psi\|_\infty+\frac{d}{2}\,\|D_{x}^2\psi\|_\infty\Big)\, \|X_u-x\|
			% solve equation
			+\int_{\|z\|\geq \delta} \Big[L\,\|z\|+\|D_{x}\psi\|_\infty\,\|h(z)\|\Big]\,|(\nu_1(dz),\ldots,\nu_d(dz))(X_u-x)|\nonumber\\
			% Lipschitz continuity
			&\quad+\frac{1}{2}\int_{\|z\|< \delta}\tr\Big[D_{x}^2\psi(t,x+\zeta_z)\,z\,z^T\Big]\,|(\nu_1(dz),\ldots,\nu_d(dz))(X_u-x)|\nonumber\\
			% Taylor expansion
			&\leq 3\,\mathcal{K}\, \Big(\|D_x\psi\|_\infty+\frac{d}{2}\,\|D_{x}^2\psi\|_\infty\Big)\, \|X_u-x\| +\int_{\|z\|\geq \delta} \Big[L\,\|z\|+\|D_{x}\psi\|_\infty\,\|h(z)\|\Big]\,|(\nu_1(dz),\ldots,\nu_d(dz))(X_u-x)|\nonumber\\
			% properties of C_h
			&\quad+\frac{d}{2}\int_{\|z\|< \delta}\left\|D_{x}^2\psi(t,x+\zeta_z)\,z\,z^T\right\|\,|(\nu_1(dz),\ldots,\nu_d(dz))(X_u-x)|\nonumber\\
			%&\leq 3\,\mathcal{K}\, \Big(\|D_x\psi\|_\infty+\frac{d}{2}\,\|D_{x}^2\psi\|_\infty\Big)\, \|X_u-x\| +\int_{\|z\|\geq \delta} \Big[L\,\underbrace{\|z\|}_{\leq C_h\,(\|z\|^2\wedge\|z\|)}+\|D_{x}\psi\|_\infty\,\underbrace{\|h(z)\|}_{\leq C_h\,(\|z\|^2\wedge\|z\|)}\Big]\,|(\nu_1(dz),\ldots,\nu_d(dz))(X_u-x)|\nonumber\\
			% properties of C_h
			%&\quad+\frac{d}{2}\int_{\|z\|< \delta}\underbrace{\left\|D_{x}^2\psi(t,x+\zeta_z)\,z\,z^T\right\|}_{\leq\|D_{x}^2\psi(t,x+\zeta_z)\|\,\|z\|^2}\,|(\nu_1(dz),\ldots,\nu_d(dz))(X_u-x)|\nonumber\\
			% Frobenius-spectral-norm inequality 
			%&\leq 3\,\mathcal{K}\, \Big(\|D_x\psi\|_\infty+\frac{d}{2}\,\|D_{x}^2\psi\|_\infty\Big)\, \|X_u-x\| \nonumber\\
			%&\quad + \Big(L+\|D_{x}\psi\|_\infty\Big)\, C_h\,\int_{\|z\|\geq \delta} (\|z\|^2\wedge\|z\|)\; |(\nu_1(dz),\ldots,\nu_d(dz))(X_u-x)| \nonumber\\
			% properties of C_h
			%&\quad +\frac{d}{2}\,\|D_x^2\psi\|_\infty \int_{\|z\|< \delta} (\|z\|^2\wedge\|z\|)\,|(\nu_1(dz),\ldots,\nu_d(dz))(X_u-x)| \nonumber\\
			% Cauchy-Schwarz
			&\leq 3\,\mathcal{K}\, \Big(\|D_x\psi\|_\infty+\frac{d}{2}\,\|D_{x}^2\psi\|_\infty\Big)\, \|X_u-x\|		
			+ 3\,\mathcal{K}\,\Big(L+\|D_{x}\psi\|_\infty\Big)\,C_h\,\|X_u-x\|
			%\nonumber\\
			+3\,\mathcal{K}\,\frac{d}{2}\,\|D_{x}^2\psi\|_\infty\,\|X_u-x\|\label{eq:Explanation2.0}\\
			&\leq \underbrace{\Big(\Big(L+2\,\|D_{x}\psi\|_\infty\Big)\,C_h+d\,\|D_{x}^2\psi\|_\infty\Big)\,3\,\mathcal{K}}_{=:\bar{C}}\,\|X_u-x\|. \label{ineq:A1}
			% C_h \geq 1
		\end{align}
		Thus, we obtain
		\begin{align}
			\mathcal{E}^x&\left(\sup_{0\leq u\leq s} \left|\mathcal{A}_{X_u}\psi(t,x)-\mathcal{A}_x\psi(t,x)\right|\right)\nonumber\\
			&\leq \mathcal{E}^x\left(  \sup_{0\leq u\leq s} \left|\mathcal{A}_{X_u}\psi(t,x)-\mathcal{A}_x\psi(t,x)\right|\,\mathbf{1}_{\{X_u\in\textnormal{S}\}}\right)+\mathcal{E}^x\left(\sup_{0\leq u\leq s}|\mathcal{A}_x\psi(t,x)|\,\mathbf{1}_{\{X_u\notin\textnormal{S}\}}\right)\nonumber\\
			& \leq \mathcal{E}^x\left(\sup_{0\leq u\leq s}\bar{C}\,\|X_u-x\|\,\mathbf{1}_{\{X_u\in\textnormal{S}\}}\right)+|\mathcal{A}_x\psi(t,x)|\,\mathcal{E}^x\left( \sup_{0\leq u\leq s}\mathbf{1}_{\{X_u\notin\textnormal{S}\}}\right).
		\end{align}
		For estimating the latter expectation, recall that $x\in\textnormal{S}$ and $x\notin\partial\textnormal{S}$, i.e., $\rho:=\textnormal{dist}(x,\mathbb{R}^d\setminus\textnormal{S})>0$, and
		\begin{equation}\label{ineq:rho}
			\rho\,\mathcal{E}^x\left( \sup_{0\leq u\leq s}\mathbf{1}_{\{X_u\notin\textnormal{S}\}}\right)
			\leq \mathcal{E}^x\left( \sup_{0\leq u\leq s}\|X_u-x\|\,\mathbf{1}_{\{X_u\notin\textnormal{S}\}}\right).
		\end{equation}
		
		\noindent Combining \eqref{ineq:A0} - \eqref{ineq:rho}, yields
		\begin{align}
			\Big| \mathcal{E}^x&\left(\sup_{0\leq u\leq s}\mathcal{A}_{X_u}\psi(t,x)\right)-\mathcal{A}_x\psi(t,x)\Big|\nonumber\\
			&\leq \bar{C}\,\mathcal{E}^x\left(\sup_{0\leq u\leq s}\|X_u-x\|\,\mathbf{1}_{\{X_u\in\textnormal{S}\}}\right)+|\mathcal{A}_x\psi(t,x)|\,\rho^{-1}\,\mathcal{E}^x\left( \sup_{0\leq u\leq s}\|X_u-x\|\,\mathbf{1}_{\{X_u\notin\textnormal{S}\}}\right)\nonumber\\
			&\leq \left(\bar{C}+|\mathcal{A}_x\psi(t,x)|\,\rho^{-1}\right)\,C_{\mathcal{K},1}\,(\|x\|+1)\,(s+s^{\frac{1}{2}})
		\end{align}
		for small enough $s\geq 0$ due to Corollary \ref{cor:TimeCont0}.
		This establishes the point-wise limit 
		\begin{equation}
			\lim_{s\downarrow0}\mathcal{E}^x\left(\sup_{0\leq u\leq s}\mathcal{A}_{X_u}\psi(t,x)\right)=\mathcal{A}_x\psi(t,x).
		\end{equation}
		Finally, since $C_{\delta,s}$ is continuous in $s$, cf. \eqref{eq:DefinitionCDeltaS}, first letting $s\downarrow0$ and then $\delta\downarrow0$ in \eqref{ineq:combined}, we obtain
		\begin{align}
			0&\leq \hat{C}\,s^{\frac{1}{2}}+2\,C_{\delta,s}-\partial_t\psi(t,x)+\mathcal{E}^x\left(\sup_{0\leq u\leq s}\mathcal{A}_{X_u}\psi(t,x)\right)\nonumber\\
			&\overset{s\downarrow0}{=}2\,C_{\delta,0}-\partial_t\psi(t,x)+\mathcal{A}_x\psi(t,x)\nonumber\\
			&\overset{\delta\downarrow0}{=}-\partial_t\psi(t,x)+\mathcal{A}_x\psi(t,x),
		\end{align}
		where we used that $C_{\delta,0}$ is a multiple of $\mathcal{K}_{\delta}(x)$, and $\lim_{\delta\downarrow0}\mathcal{K}_\delta(x)=0$ for all $x\in\mathbb{R}^d$ due to Assumption \ref{con:C}.
		Thus, the value function $v$ satisfies \eqref{eq:def-visc2}.
	\end{proof}
	
    \section{Uniqueness of the Viscosity Solution} \label{sec:Uniqueness}
We now study uniqueness of viscosity solutions of the PIDE \eqref{eq:PIDE} - \eqref{eq:PIDE-initial}.
	Unfortunately, comparison and uniqueness results for viscosity solutions tend to become more involved when considering an arbitrary domain.
	Even for the choice $\textnormal{S}=\mathbb{R}^d$, the associated PIDE \eqref{eq:PIDE} fails to satisfy the necessary assumptions, which allow to apply comparison and uniqueness results, cf. eg. \cite{barles_second-order_2008}, \cite[pp. 30 - 31]{neufeld_nonlinear_2016}, \cite[Chapter 2]{hollender_levy-type_2016}, \cite{jakobsen_maximum_2006}.
	This is due to our definition in \eqref{eq:theta1} - \eqref{eq:theta3}.
	More precisely, since we do not have that $\Theta(x)\subseteq\mathbb{R}^d\times\mathbb{S}_+\times\mathfrak{L}_+$ for all $x\in\textnormal{S}$, the degenerate ellipticity condition is not satisfied.
	One possible solution would be to restrict the supremum in the definition of $\mathcal{A}_y$ in \eqref{eq:def-nonlinear-A} to parameters $(\beta,\alpha,\nu)\in\Theta$ such that $\alpha(y)\in\mathbb{S}_+$ and $\nu(y)\in\mathfrak{L}_+$.
	This is in line with the probabilistic reasoning behind the construction since the differential characteristics $(b,a,k)$ takes values in $\mathbb{R}^d\times\mathbb{S}_+\times\mathfrak{L}_+$, 
	but it would cause problems in the first estimation in \eqref{ineq:A1}.
	In order to circumvent these issues, we consider the canonical state space $\textnormal{S}=\mathbb{R}^d$ and an alternative parameter function by following the approach of \cite{fadina_affine_2019}.
	
	For $(\beta,\alpha,\nu)\in( \mathbb{R}^d)^{d+1}\times\mathbb{S}^{d+1}\times\mathfrak{L}^{d+1}$, define  
	\begin{align}
		\hat{\beta}(x)&:=\beta_0+(\beta_1,\ldots,\beta_d)\,x,\\
		\hat{\alpha}(x)&:=\alpha_0+(\alpha_1,\ldots,\alpha_d)\,x^+,\\
		\hat{\nu}(x)&:=\nu_0 \label{eq:UniquenessCharacteristicThirdComponent}
	\end{align}
	for all $x\in\mathbb{R}^d$, where $x^+:=x\vee0\in\mathbb{R}^d$ coordinate-wise.
	Analogously, for any subset $\Theta\subseteq (\mathbb{R}^d)^{d+1}\times\mathbb{S}^{d+1}\times\mathfrak{L}^{d+1}$ and $x\in\mathbb{R}^d$, define
	\begin{equation} \label{eq:definitionHatTheta}
		\hat{\Theta}(x):=\Big\{ (\hat{\beta}(x),\hat{\alpha}(x),\hat{\nu}(x))\,:\, (\beta,\alpha,\nu)\in\Theta \Big\}. 
	\end{equation}
	As in \eqref{eq:NLAJD} we define the family $\lbrace \mathcal{P}_x(\hat{\Theta}) \rbrace_{x \in \mathbb{R}}$ by
	\begin{equation}\label{eq:NLAJDFadina}
		\mathcal{P}_x(\hat{\Theta}):=\Big\{P\in\mathfrak{P}_{sem}^{ac}(\Omega)\,:\,\;P(X_0=x)=1;\:(b^P\!, a^P\!, k^P)\in \hat{\Theta}(X)\text{ $dt\otimes dP$-a.e.}\Big\},
	\end{equation}
	and let $\{\hat{\mathcal{E}}^x\}_{x\in\mathbb{R}^d}$ denote the associated sublinear expectations.
	\begin{rem}
	    Note that the function $\hat{\nu}(x)$ is constant. This is necessary in order to prove that the corresponding value function is a unique viscosity solution of the associated PIDE. 
	\end{rem}
	
	\begin{thm}\label{thm:hat}
		Let $\Theta\subseteq (\mathbb{R}^d)^{d+1}\times\mathbb{S}_+^{d+1}\times\mathfrak{L}_+^{d+1}$ satisfy Condition \ref{con:C}. Assume that there exists a constant $\hat{L}>0$ such that for all $x,y\in\mathbb{R}^d$
		\begin{equation}\label{con:sqrt-alpha}
			\sup_{(\alpha, \beta, \nu_0) \in \Theta} \left\| \sqrt{\hat{\alpha}(x)} -\sqrt{\hat{\alpha}(y)} \right\| \leq \hat{L}\,\|x-y\|,
		\end{equation}
		where $\sqrt{\hat{\alpha}(\cdot)}$ denotes the unique square root of the matrix $\hat{\alpha}(\cdot)$.
		Let $\phi\in\textnormal{Lip}_b(\mathbb{R}^d)$.
		Define the value function 
		\begin{equation}\label{eq:def-value-function-hat}
			\hat{v}(t,x):=\hat{\mathcal{E}}^x(\phi(X_t)).
		\end{equation}
		If $\hat{v}$ is continuous, then $\hat{v}$ is the unique viscosity solution of the PIDE
		\begin{align}\label{eq:PIDE-hat}
			\partial_t v(t,x) -\hat{\mathcal{A}}_x v(t,x)&=0 
		\end{align}
		in $\mathbb{R}_+\times\mathbb{R}^d$ such that the initial condition \eqref{eq:PIDE-initial} holds on $\mathbb{R}^d$, where the operator $\hat{\mathcal{A}}_y$ is given by
		\begin{align}
			\hat{\mathcal{A}}_y f(x)&:=\sup_{(\beta,\alpha,\nu)\in\Theta} \Bigg\{D_x f(x)^T\,\hat{\beta}(y)+\frac{1}{2}\,\tr\Big[D_{x}^2 f(x)\,\hat{\alpha}(y)\Big]
			%\nonumber\\
			%&\qquad\qquad\qquad
			+\int_{\mathbb{R}^d} \left[f(x+z)-f(x)-D_x f(x)^T\,h(z)\right]\,\hat{\nu}(y;dz) \Bigg\},
		\end{align}
		with $h(z):=z\,\mathbf{1}_{\{\|z\|\leq 1\}}(z)$ for all $z\in\mathbb{R}^d$.
	\end{thm}
	As the proofs in Sections \ref{sec:Dynamic} - \ref{sec:PIDE} in general do not depend specifically on the form of the functions $\beta, \alpha, \nu$, the main results of these two sections carry over to this setting. Thus, the following results hold.
	\begin{lem}\label{lem:dynamic-hat}
		Let $\Theta\subseteq (\mathbb{R}^d)^{d+1}\times\mathbb{S}^{d+1}\times\mathfrak{L}^{d+1}$ satisfy Assumption \ref{con:C}.
		Then the family $\{ \mathcal{P}_x(\hat{\Theta})\}_{x\in\mathbb{R}^d}$ is amenable to dynamic programming in the sense of Proposition \ref{prop:DynamicProgramming}, and the associated family of sublinear expectations $\{ \hat{\mathcal{E}}^x\}_{x\in\mathbb{R}^d}$ satisfies the Markov property from Lemma \ref{lem:SubMarkov}.
	\end{lem}
	\begin{proof}
		By Condition \ref{con:C}, $\Theta\subseteq(\mathbb{R}^d)^{d+1}\times\mathbb{S}^{d+1}\times\mathfrak{L}^{d+1}$ is non-empty and closed.
		Hence, the set 
		\begin{equation} \label{eq:SetBorelFadina}
			\left\{(x,b,a,k)\in\mathbb{R}^d\times\mathbb{R}^d\times\mathbb{S}\times\mathfrak{L}\,:\,(b,a,k)\in\hat{\Theta}(x)\right\}
		\end{equation}
		is Borel.
		This follows by the same arguments as in the proof of Lemma \ref{lem:Borel} since they solely rely on the separability of the space $(\mathbb{R}^d)^{d+1} \times \mathbb{S}^{d+1}\times\mathfrak{L}^{d+1}$, cf. Remark \ref{rem:Borel}. 
		The dynamic programming principle and the Markov property follow as in the proofs of Proposition \ref{prop:DynamicProgramming} and Lemma \ref{lem:SubMarkov}.
	\end{proof}
	\begin{lem}\label{lem:existence-hat}
		Let $\Theta\subseteq (\mathbb{R}^d)^{d+1}\times\mathbb{S}^{d+1}\times\mathfrak{L}^{d+1}$ satisfy Assumption \ref{con:C} and  $\phi\in\textnormal{Lip}_b(\mathbb{R}^d)$. Assume that the value function $\hat{v}$ from \eqref{eq:def-value-function-hat} is continuous. 
		Then $\hat{v}$ is a viscosity solution of \eqref{eq:PIDE-hat} in $\mathbb{R}_+\times\mathbb{R}^d$, and satisfies the initial condition \eqref{eq:PIDE-initial} on $\mathbb{R}^d$.
	\end{lem}
	\begin{proof}
		Note that in the proofs of Corollaries \ref{cor:lin_bound}, \ref{cor:TimeCont0} and Lemmas \ref{lem:TimeCont0}, \ref{lem:integrability} - \ref{lem:JointContinuity}, the specific form of the functions $\beta,\alpha,\nu$ is not used.
		Instead, these results follow from the dynamic programming principle, the sublinear Markov property, and the inequalities in Assumption \ref{con:C}.
		
		Since $\Theta\subseteq (\mathbb{R}^d)^{d+1}\times\mathbb{S}^{d+1}\times\mathfrak{L}^{d+1}$ satisfies Condition \ref{con:C}, we have
		\begin{align}\label{con:U1}
			\lim_{\delta\downarrow 0}\sup_{(\beta,\alpha,\nu)\in\Theta}\int_{\|z\|< \delta} \|z\|^2\,\hat{\nu}(x;dz) &
			\equiv {\lim_{\delta\downarrow 0}\sup_{(\beta,\alpha,\nu)\in\Theta}\int_{\|z\|< \delta} \|z\|^2\,{\nu_0}(dz)} \\
			& { = \lim_{\delta\downarrow 0}\sup_{(\beta,\alpha,\nu)\in\Theta}\int_{\|z\|< \delta} \|z\|^2\,{\nu}(0;dz)
				=}
			\lim_{\delta\downarrow 0} \mathcal{K}_\delta(0)
			=0.
		\end{align}
		Moreover,
		\begin{equation}\label{con:U2}
			\sup_{x\in\mathbb{R}^d} \frac{\| \hat{\beta}(x) \| }{\|x\|+1}
			=\trinorm{\beta} ,
			\qquad
			\sup_{x\in\mathbb{R}^d} \frac{\| \hat{\alpha}(x) \| }{\|x\|+1}
			\leq \trinorm{\alpha},
			\qquad 
			\sup_{x\in\mathbb{R}^d} \frac{\| \hat{\nu}(x) \|}{\|x\|+1} 
			\leq \trinorm{\nu}
		\end{equation}
		for all $(\beta,\alpha,\nu)\in\Theta$.
		Hence,  Corollaries \ref{cor:lin_bound}, \ref{cor:TimeCont0} and Lemmas \ref{lem:TimeCont0}, \ref{lem:integrability} - \ref{lem:JointContinuity} are also valid for the family $\{\hat{\mathcal{E}}^x\}_{x\in\mathbb{R}^d}$ of sublinear expectations and the value function $\hat{v}$.
		
		By carefully examining the arguments in the proof of Proposition \ref{prop:Existence} we can see that the specific form of the functions $\alpha, \beta, \nu$ is only used in \eqref{ineq:A1}.
		Thus, proving that $\hat{v}$ is a viscosity solution of the PIDE \eqref{eq:PIDE-hat} in $\mathbb{R}_+\times\mathbb{R}^d$ boils down to establishing the inequality \eqref{ineq:A1} for the operator $\hat{\mathcal{A}}_x$.
		For this purpose, note that $\|x^+-y^+\|\leq \|x-y\|$ for all $x,y\in\mathbb{R}^d$.
		Hence, we obtain
		\begin{align}
			&| \hat{\mathcal{A}}_{X_u}\psi(t,x) - \hat{\mathcal{A}}_{x}\psi(t,x) |\nonumber\\
			% &\leq \Bigg| \sup_{(\beta,\alpha,\nu)\in\Theta}\Bigg\{D_{x}\psi(t,x)^T\, (\beta_1,\ldots,\beta_d)\,(X_u-x)+\frac{1}{2}\,\tr\Big[D_{x}^2\psi(t,x)\,(\alpha_1,\ldots,\alpha_d)\,(X_u^+-x^+)\Big] \Bigg\} \Bigg| \nonumber\\
			% supremum inequality
			&\leq 
			\sup_{(\beta,\alpha,\nu)\in\Theta} \Bigg\{ \Big| D_{x}\psi(t,x)^T\, (\beta_1,\ldots,\beta_d)\,(X_u-x)\Big| +\frac{1}{2}\,\left|\tr\Big[D_{x}^2\psi(t,x)\,(\alpha_1,\ldots,\alpha_d)\,(X_u^+-x^+)\Big]\right| \Bigg\} \nonumber \\
			% triangle inequality
			& \leq  
			%\expl{} \expl{3\, \mathcal{K}\,\|D_x\psi\|_\infty\,\|X_u-x\| + 3\,\mathcal{K}\,\frac{d}{2}\,\|D_{x}^2\psi\|_\infty\,\|X_u-x\|}\label{eq:Explanation3.0}\\
			% inequality for linear part for \beta and \alpha
			%\expl{} & \expl{=}
			3\,\mathcal{K}\, \Big(\|D_x\psi\|_\infty+\frac{d}{2}\,\|D_{x}^2\psi\|_\infty\Big)\, \|X_u-x\|,
			% solve equation
		\end{align}
		where we used Assumption \eqref {con:C} %in \eqref{eq:Explanation3.0}
		.
	\end{proof}
	\begin{proof}[Proof of Theorem \ref{thm:hat}]
		Due to Lemmas \ref{lem:dynamic-hat} and \ref{lem:existence-hat}, we are only left with proving the uniqueness of the viscosity solution $\hat{v}$.
		This in turn follows from Assumption \ref{con:HJB} together with the comparison result for Hamilton-Jacobi-Bellman equations from \cite{hollender_levy-type_2016}, see the Appendix.
		
		The PIDE \eqref{eq:PIDE-hat} clearly has the form as in \eqref{eq:AppendixPIDEViscosity} since $\hat{\alpha}(x)\in\mathbb{S}_+$ and its unique square root exists.
		%Note that the growth and the monotonicity conditions are trivial.
		For the boundedness condition on the local part, observe that for every $x\in\mathbb{R}^d$,
		\begin{align}
			\sup_{(\beta,\alpha,\nu)\in\Theta} \bigg\{ \|\hat{\beta}(x)\| + \left\| \sqrt{\hat{\alpha}(x)}\right\| \bigg\} 
			%& = \sup_{(\beta,\alpha,\nu)\in\Theta} \Big\{ \|\hat{\beta}(x)\| + \sqrt{\| \hat{\alpha}(x)\|} \Big\} \nonumber\\
			% root of eigenvalues is taken, hence this carries over
			& \leq \sup_{(\beta,\alpha,\nu)\in\Theta} \Big\{ \|\hat{\beta}(x)\| + {\| \hat{\alpha}(x)\|} + 1 \Big\} 
			\nonumber \\
			% x < x^2 +1
			& 
			\leq 
			\mathcal{K}\,(\|x\|+1) + 1
			% boundedness condition
			<\infty
		\end{align}
		due to \eqref{con:U2} and Assumption \ref{con:C}.
		Analogously, the non-local part is bounded
		\begin{equation}
			\sup_{(\beta,\alpha,\nu)\in\Theta} \int_{\mathbb{R}^d} \|z\|^2\wedge \|z\|\;\hat{\nu}(x;dz) \equiv  \sup_{(\beta,\alpha,\nu)\in\Theta} \int_{\mathbb{R}^d} \|z\|^2\wedge \|z\|\; \nu(0,dz)
			\leq \mathcal{K} < \infty
		\end{equation}
		by \eqref{eq:UniquenessCharacteristicThirdComponent}.
		The tightness condition follows from \eqref{con:U1}.
		By Corollary \ref{cor:lin_bound} and \eqref{con:sqrt-alpha},
		\begin{equation}
			\sup_{(\alpha, \beta, \nu_0) \in \Theta} \Bigg\{\left\| \hat{\beta}(x) - \hat{\beta}(y)\right\| + \left\| \sqrt{\hat{\alpha}(x)} -\sqrt{\hat{\alpha}(y)} \right\| \Bigg\}
			\leq \left(3\,\mathcal{K}+\hat{L}\right)\,\|x-y\|
		\end{equation}
		for all $x,y\in\mathbb{R}^d$, which establishes the continuity condition.
		
		Applying Corollary \ref{cor:ComparisonPrinciple} immediately yields that $\hat{v}$ is the unique viscosity solution of \eqref{eq:PIDE-hat} in $\mathbb{R}_+\times\mathbb{R}^d$ such that the initial condition \eqref{eq:PIDE-initial} holds on $\mathbb{R}^d$.
	\end{proof}
	
	\begin{rem}
		Note that although the results in Sections \ref{sec:Dynamic} and \ref{sec:PIDE} do not depend on the specific form of the affine functions $\beta, \alpha, \nu$% or respectively $\hat{\beta}, \hat{\alpha}, \hat{\nu}$
		, the results in Lemmas \ref{lem:jump-exit}, \ref{lem:tau-infinite} and Corollaries \ref{cor:constant-outside}, \ref{cor:trivial_uncertainty_subset} in Section \ref{sec:setup} do.
		By choosing the state space $\textnormal{S}=\mathbb{R}^d$ for the non-linear affine process with uncertainty subsets $\{\mathcal{P}_x(\hat{\Theta})\}_{x\in\textnormal{S}}$, we avoid further technicalities.
	\end{rem}

    \section{Examples}\label{sec:examples}
	\delete{	\subsection{Parameter Sets of admissible Parameters}\label{sec:admissible-parameters}}
	In the following, we give an example of a non-linear affine process.
	For the sake of simplicity, we consider one-dimensional processes with canonical state space $\textnormal{S}=\mathbb{R}$ or $\textnormal{S}=\mathbb{R}_+$.
	\delete{Recall that we have constructed non-linear affine processes by allowing the differential characteristics to evolve in $\Theta(X)$.
		This construction admits a straightforward interpretation and application.
		If we assume a stochastic phenomenon (e.g. short rates, volatility or credit default) is described by an affine process and we try to estimate its parameters, then the estimators are fraught with statistical uncertainty
		In order to account for this uncertainty, we may consider confidence intervals for each interval instead of one single estimator.
		The thus obtained model is much more robust than the classical model with fixed parameters.
		From this statistical point of view, it makes sense to define $\Theta\subseteq\mathbb{R}^2\times\mathbb{R}^2\times\mathfrak{L}^2$ such that each parameter evolves in some interval (or a convex, compact set in the case of Lévy measures).}
	
	\delete{In this section, we apply that approach to three broad classes of affine processes before turning to parameter sets $\Theta$ with non-admissible parameters in Section \ref{sec:non-admissible-parameters}.
		We start with two classes of affine diffusion processes, the first of which can be found in a slightly different presentation in \cite{fadina_affine_2019}, too.}
	We focus here on a class of pure jump processes with admissible parameters.
	For affine diffusion processes we refer to \cite{fadina_affine_2019} and the examples therein.
	\delete{The third class contains pure jump processes, and is thus a noticeable contribution of this paper.}
	\delete{As continuous image of a compact set, all presented choices of $\Theta$ are closed and non-empty by construction.
		In fact, by applying suitable transformations, $\Theta$ can written as the Cartesian product of closed intervals and convex, compact sets of Lévy measures.
		Thus, the dynamic programming principle holds and the associated uncertainty sets $\{\mathcal{P}_x(\Theta)\}_{x\in\textnormal{S}}$ define a non-linear affine process.
		Further note that for each particular choice of $\Theta$ in this section, the uncertainty sets are non-empty, i.e., $\mathcal{P}_x(\Theta)\neq\emptyset$ for all $x\in\mathbb{R}$.
		As explained in Remark \ref{rem:non-empty}, this follows from \cite[Theorem 2.7]{duffie_affine_2003}, since each $\Theta$ contains admissible parameters in the sense of \cite[Definition 2.6]{duffie_affine_2003}.}
	For the sake of convenience, we state here the definition of an admissible parameter for $d=1$, cf. \cite[Definition 2.6]{duffie_affine_2003}, \cite[pp. 19 - 21]{keller-ressel_affine_2008}. 
	
	\begin{defn}
		A parameter $(\beta_0,\beta_1,\alpha_0,\alpha_1,\nu_0,\nu_1)\in\mathbb{R}^2\times\mathbb{R}^2\times\mathfrak{L}^2$ is called \emph{admissible} for the state space $\textnormal{S}=\mathbb{R}_+$ if
		\begin{equation}
			\beta_0\in\mathbb{R}_+,\; \beta_1\in\mathbb{R},\; \alpha_0=0,\; \alpha_1\in\mathbb{R}_+,\;  \nu_0, \nu_1\in\mathfrak{L}_+ \text{ with}\nonumber
		\end{equation}
		\begin{equation}
			\text{supp}(\nu_0), \text{supp}(\nu_1)\subseteq\mathbb{R}_+,\;\int_0^\infty z\wedge 1\,(\nu_0(dz)+\nu_1(dz))<\infty; 
		\end{equation}
		and for the state space $\textnormal{S}=\mathbb{R}$ if
		\begin{equation}
			\beta_0\in\mathbb{R},\;\beta_1\in\mathbb{R},\; \alpha_0\in\mathbb{R}_+,\; \alpha_1=0,\; \nu_0\in\mathfrak{L}_+,\; \nu_1=0.
		\end{equation}
	\end{defn}

	Consider a compound Poisson process given by
	\begin{equation}
		X_t:=\sum_{n=1}^{N_t}Y_n,
	\end{equation}
	where $(Y_n)_{n\in\mathbb{N}}$ is a sequence of i.i.d. random variables with law $\nu$, and $N=(N_t)_{t\geq 0}$ is a Poisson process with intensity $\lambda >0$ independent of $(Y_n)_{n\in\mathbb{N}}$, cf. \cite[p. 449]{andersen_handbook_2009}.
	For simplicity, assume that $\nu\in\mathfrak{L}_+$, i.e., $\nu(\{0\})=0$ and $\int ( z^2\wedge 1)\,d\nu(dz)<\infty$.
	Fix a truncation function $h:\mathbb{R}\rightarrow\mathbb{R}$.
	Then $X$ admits the canonical decomposition
	\begin{equation}
		X_t = \underbrace{\lambda\,t\int_{\mathbb{R}} h(z)\,\nu(dz)}_{=B_t} + \underbrace{\sum_{n=1}^{N_t} h(Y_n)-\lambda\,t \int_{\mathbb{R}} h(z)\,\nu(dz)}_{=\,^j\!M_t=M_t} + \underbrace{\sum_{n=1}^{N_t} \left(Y_n - h(Y_n)\right)}_{=J_t}.
	\end{equation}
	That is, $X$ is a Lévy process and thus an affine process. 
	All parameters are zero except for $\beta_0=\lambda\int h(z)\,\nu(dz)$ and $\nu_0=\lambda\,\nu$.
	If $\text{supp}(\nu)\subseteq\mathbb{R}_+$ and $\int_0^\infty ( z \wedge 1 )\,\nu(dz)<\infty$, the state space of $X$ is $\textnormal{S}=\mathbb{R}_+$, or $\textnormal{S}=\mathbb{R}$ otherwise.
	Now, fix some constants $0\leq \underline{\lambda}\leq\overline{\lambda}$ and a convex, compact set $\emptyset\neq L\subseteq\mathfrak{L}_+$.
	Set
	\begin{align}\label{eq:theta-compound-poisson}
		\Theta &:= \left\{ \left(\lambda\, \int_{\mathbb{R}} h(z)\,\nu(dz),0,0,0,\lambda\,\nu,0 \right)\in\mathbb{R}^2\times\mathbb{R}^2\times\mathfrak{L}^2\,:\,\lambda\in[\underline{\lambda},\overline{\lambda}],\;\nu\in L\right\}\nonumber\\
		&=\left[\underline{\lambda}\,\min_{\nu\in L}\int_{\mathbb{R}}h(z)\,\nu(dz),\overline{\lambda}\,\max_{\nu\in L}\int_{\mathbb{R}}h(z)\,\nu(dz)\right]\times\{0\}\times\{0\}\times\{0\}\times L^{\lambda}\times\{0\},
	\end{align}
	where $L^\lambda:=\{\lambda\,\nu\,:\,\lambda\in[\underline{\lambda},\overline{\lambda}],\;\nu\in L\}$, which is convex and compact.
	Since $L$ is compact, the maximum and minimum in \eqref{eq:theta-compound-poisson} are attained, and all intermediate values can be obtained choosing suitable $\lambda\,\nu\in L^\lambda$ due to its convexity.
	As in the classical setting, we choose $\textnormal{S}=\mathbb{R}_+$ if $\text{supp}(\nu)\subseteq\mathbb{R}_+$ and $\int_0^\infty \left( z\wedge 1 \right)\,\nu(dz)<\infty$ for all $\nu\in L$, or $\textnormal{S}=\mathbb{R}$ otherwise.
	We call $X$ with associated uncertainty sets $\{\mathcal{P}_x(\Theta)\}_{x\in\textnormal{S}}$ a \emph{non-linear compound Poisson process} with state space $\textnormal{S}$.
	
	Note that the linear parameters of this non-linear affine processes are zero, i.e., it is a non-linear Lévy process as in \cite{neufeld_nonlinear_2016}.
	In the classical setting, we can generalise the compound Poisson process to an affine pure jump process by replacing the constant intensity $\lambda$ by an affine intensity $\lambda(x):=\lambda_0+\lambda_1\,x>0$, cf. \cite[p. 1349]{duffie_transform_2000}.
	That is, for the state space $\textnormal{S}=\mathbb{R}_+$, we can consider
	\begin{align} \label{eq:theta-generalised-compound-poisson}
		\Theta&:=\Bigg\{ \left(\lambda_0\,t\int_{\mathbb{R}} h(z)\,\nu(dz), \lambda_1\,t\int_{\mathbb{R}} h(z)\,\nu(dz),0,0,\lambda_0\,\nu,\lambda_1\,\nu \right)\in\mathbb{R}^2\times\mathbb{R}^2\times\mathfrak{L}^2\,:\,\nonumber\\ 
		&\qquad\qquad\qquad\qquad\qquad\qquad\qquad\qquad\qquad\qquad\quad
		\lambda_0\in[\underline{\lambda_0},\overline{\lambda_0}], \lambda_1\in [\underline{\lambda_1},\overline{\lambda_1}],\;\nu\in L\Bigg\},
	\end{align}
	where $0\leq \underline{\lambda_0}\leq\overline{\lambda_0}$, $0\leq\underline{\lambda_1}\leq\overline{\lambda_1}$ are some constants, and $\emptyset\neq L\subseteq\mathfrak{L}_+$ is convex and compact with $\text{supp}(\nu)\subseteq\mathbb{R}_+$ and $\int_0^\infty \left(z\wedge 1\right)\,\nu(dz)<\infty$ for all $\nu\in L$.
	We call $X$ with associated uncertainty sets $\{\mathcal{P}_x(\Theta)\}_{x\in\textnormal{S}}$ a \emph{non-linear generalised compound Poisson process} with state space $\textnormal{S}=\mathbb{R}_+$.
	
	Clearly, $\Theta$ in \eqref{eq:theta-compound-poisson}, respectively in \eqref{eq:theta-generalised-compound-poisson}, is closed and non-empty, such that the dynamic programing principle holds.
	Further, note that for $\Theta$ in \eqref{eq:theta-compound-poisson}, respectively \eqref{eq:theta-generalised-compound-poisson}, the uncertainty sets are non-empty, i.e., $\mathcal{P}_x(\Theta)\neq\emptyset$ for all $x\in\mathbb{R}$ due to the admissibility of the parameters, cf. Remark \ref{rem:non-empty}. 
	Moreover, it is obvious that for the non-linear compound Poisson process the definition of $\Theta(x)$ in \eqref{eq:Theta} and the one of $\hat{\Theta}(x)$ in \eqref{eq:definitionHatTheta} coincide.
	Thus, by Theorem \ref{thm:hat} the value function $v(t,x):=\mathcal{E}^x(\phi(X_t))$ is the unique viscosity solution of the corresponding PIDE \eqref{eq:PIDE}-\eqref{eq:PIDE-initial}.
	Unfortunately, for the non-linear generalised compound Poisson process we only get an existence result for the solution of the associated PIDE.
	
	\delete{In the light of the interpretation we have given at the beginning of Chapter \ref{ch:Examples}, it seems natural to choose the state space $\textnormal{S}$ and the parameter set $\Theta$ such that every $\theta\in\Theta$ is admissible for the state space $\textnormal{S}$.}
	\begin{rem}
		Note that in Sections \ref{sec:setup} - \ref{sec:PIDE} we have not required $\Theta$ to be a set of admissible parameters but allowed it to be an arbitrary non-empty, closed subset of $(\mathbb{R}^d)^{d+1}\times\mathbb{S}^{d+1}\times\mathfrak{L}^{d+1}$.
		In fact, it is possible to choose $\textnormal{S}$ and $\Theta$ such that no $\theta\in\Theta$ is admissible for the state space $\textnormal{S}$ but $\mathcal{P}_x(\Theta)$ is non-empty for all $x\in\mathbb{R}^d$.
		This is due to our definition of $\mathcal{P}_x(\Theta)$ in \eqref{eq:NLAJD}.
		More precisely, \delete{this is due to the fact that} allow the differential characteristics to evolve in $\Theta(X)$, {i.e.,} we ask for
		\begin{equation}
			(b^P,a^P,k^P)\in\Theta(X)\quad dt\otimes dP\text{-a.e.}
		\end{equation}
		rather than to require that 
		\begin{equation} \label{eq:DifferentAssumption}
			\exists \theta\in\Theta:\,\quad (b^P,a^P,k^P)=\theta(X)\; dt\otimes dP\text{-a.e.}
		\end{equation}
		Assumption \eqref{eq:DifferentAssumption} corresponds to parameter uncertainty in the narrow sense and to the set $\check{\mathcal{P}}_x(\Theta)$ as defined in \eqref{eq:parameter-uncertainty}.
		The set $\check{\mathcal{P}}_x(\Theta)$ being non-empty is equivalent to $\Theta$ containing some admissible parameters.
		More precisely, \cite[Theorem 2.7]{duffie_affine_2003} yields that for every $P\in\check{\mathcal{P}}_x(\Theta)$ the process $X$ is affine with admissible parameters.
		On the contrary, if $\theta\in\Theta$ is admissible, then there exists a $P\in\check{\mathcal{P}}_x(\Theta)$ such that $X$ is affine with parameters $\theta$. This approach is especially suitable to take into account statistical uncertainty for parameter estimation of a stochastic phenomenon (e.g. short rates, volatility or credit default) by considering confidence intervals instead of one single estimator. On the other hand, our construction is more general and allows for more flexibility.
	\end{rem}

	\appendix
	\section{Appendix} \label{appendix}
	
	In this Appendix we summarize some important results we used in the previous sections. First, we state the Burkholder-Davis-Gundy inequality for $\mathbb{R}^d$-valued  càdlàg local martingales. Second, we summarize some properties of truncation functions and conclude with a comparison result for PIDEs. 

    \begin{lem}\label{lem:BDG}
    	Let $p\geq 1$.
    	There exists a constant $0< C_{p,d}<\infty$ such that for all $\mathbb{R}^d$-valued càdlàg local martingales $M$ and $t\geq 0$,
    	\begin{equation}\label{eq:BDG}
    		E\left[\sup_{0\leq s\leq t}\|M_s-M_0\|^p\right]\leq C_{p,d}\,E\left[\,\|[M]_t\|^{\frac{p}{2}}\,\right].
    	\end{equation}
    \end{lem}

    \begin{proof}
    	This follows from the Burkholder-Davis-Gundy inequality for real-valued local martingales, cf. e.g. \cite[Theorem IV.48]{protter_stochastic_2004}.
    	Let $C_p$ denote the constant for the one-dimensional case, and $M=(M^1,\ldots,M^d)^T$ be an $\mathbb{R}^d$-valued càdlàg local martingale.
    	By the convexity of $x\mapsto x^p$ on $\mathbb{R}_+$ and the triangle inequality, we obtain that
    	\begin{align}
    		E\left[\sup_{0\leq s\leq t} \|M_s\|^p \right] 
    	%	& = E\left[\sup_{0\leq s\leq t} \left(\sum_{i=1}^{d} |M_s^i|^2\right)^{\frac{p}{2}}	\right] \nonumber \\
    		% definition Euclidean norm
    	%	 & \leq E\left[\sup_{0\leq s\leq t} \left(\sum_{i=1}^{d} |M_s^i|\right)^p \right] \nonumber\\
    		% a^2+b^2 < (|a|+|b|)^2
    	%    & \leq E\left[  \left(\sum_{i=1}^{d} \sup_{0\leq s\leq t} |M_s^i|\right)^p \right] \nonumber\\
    		% sup sum < sum sup
    		&\leq E\left[\left(\sum_{i=1}^d\sup_{0\leq s\leq t}|M_s^i|\right)^p\,\right]\nonumber\\ 
    		% triangle-inequality
    	%    &\leq d^{p-1}\,\sum_{i=1}^d E\left[\sup_{0\leq s\leq t}|M_s^i|^p\right] \nonumber\\
    		% convexity of x->x^p
    		&\leq d^{p-1}\,\sum_{i=1}^d C_p\,E\left[([M^i]_t)^{\frac{p}{2}}\right]\nonumber\\ 
    		% BDG inequality for 1-dim.
    		&\leq d^p\,C_p\,\max_{1\leq i\leq d}E\left[([M^i]_t)^{\frac{p}{2}}\right]\nonumber\\ 
    		% maximum 
    		&= d^p\,C_p\,\max_{1\leq i\leq d}E\left[\left(e_i^T\,[M]_t\,e_i\right)^{\frac{p}{2}}\right]\nonumber\\ 
    		% definition of [M] coordinate-wise
    		&\leq d^p\,C_p\,E\left[\|[M]_t\|^{\frac{p}{2}}\right], 
    		% operator norm/ Cauchy-Schwarz
    	\end{align}
    	where $e_i$ denotes the $i$-th canonical basis vector.
    \end{proof}

	\begin{cor}\label{cor:truncation}
    	Let $h:\,\mathbb{R}^d\rightarrow\mathbb{R}^d$ be a truncation function. Then there exists a constant $C_h>0$ such that,
    	\begin{equation}
    		\|z-h(z)\| \leq C_h\,(\|z\|^2\wedge\|z\|)\quad\text{and}\quad \|h(z)\|\leq C_h\,(\|z\|\wedge 1) \qquad \text{for all $z\in\mathbb{R}^d$};
    	\end{equation}
    	\begin{equation}
    	    \|z\| \leq C_h\,(\|z\|^2\wedge\|z\|) \qquad\text{ outside some neighbourhood of zero.}
    	\end{equation}
    \end{cor}
    
    \begin{proof}
    	By the definition of $h$, there exists a $0<\delta\leq
    	1$ such that $h(z)=z$ for all $z$ with $\|z\|\leq\delta$.
    	Fix $\delta>0$. Then clearly $\delta \leq \|h\|_\infty$, and
    	\begin{align}
    		\|z-h(z)\|
    		&=\|z-h(z)\|\,\mathbf{1}_{\{\|z\|\geq \delta\}}(z)\leq \|z\|\,\mathbf{1}_{\{\|z\|\geq \delta\}}(z)+\|h\|_\infty \,\mathbf{1}_{\{\|z\|\geq \delta\}}(z),
    	\end{align}
    	\begin{equation}
    		\|h(z)\|\leq\|z\|\,\mathbf{1}_{\{\|z\|< \delta\}}(z)+\|h\|_\infty\,\,\mathbf{1}_{\{\|z\|\geq \delta\}}(z).
    	\end{equation}
    	Further, observe that
    	\begin{align}
    		\|z\|\,\mathbf{1}_{\{\|z\|\leq\delta\}}(z)
    		&\leq \underbrace{(\|z\|\wedge 1)}_{= \|z\|}\,\underbrace{\frac{\|h\|_\infty}{\delta}}_{\geq1}\,\underbrace{\frac{1}{\delta}}_{\geq1}\,\mathbf{1}_{\{\|z\|\leq\delta\}}(z)=(\|z\|\wedge 1)\,\delta^{-2}\,\|h\|_\infty\,\mathbf{1}_{\{\|z\|\leq\delta\}}(z),
    	\end{align}
    	\begin{align}
    		\|z\|\,\mathbf{1}_{\{\|z\|\geq \delta\}}(z)
    		&\leq \|z\|\,\underbrace{\frac{\|z\|\wedge1}{\delta}}_{\geq1}\,\underbrace{\frac{\|h\|_\infty}{\delta}}_{\geq 1}\,\mathbf{1}_{\{\|z\|\geq \delta\}}(z)= (\|z\|^2\wedge \|z\|)\,\delta^{-2}\,\|h\|_\infty\,\mathbf{1}_{\{\|z\|>\delta\}}(z),
    	\end{align}
    	\begin{align}
    		\|h\|_\infty\,\mathbf{1}_{\{\|z\|\geq \delta\}}(z)
    		&\leq \underbrace{\frac{\|z\|\wedge1}{\delta}}_{\geq 1}\,\|h\|_\infty\,\mathbf{1}_{\{\|z\|\geq \delta\}}(z)\leq 
    		\begin{cases}
    			(\|z\|\wedge 1)\,\delta^{-2}\,\|h\|_\infty\,\mathbf{1}_{\{\|z\|>\delta\}}(z)\\
    			(\|z\|^2\wedge \|z\|)\,\delta^{-2}\,\|h\|_\infty\,\mathbf{1}_{\{\|z\|>\delta\}}(z),
    		\end{cases}
    	\end{align}
    	where we used $\frac{\|z\|}{\delta}, \frac{1}{\delta}\geq 1$ in the last step. Hence, $C_h=\delta^{-2}\,\|h\|_\infty$ is a possible choice.
    \end{proof}

	%\section{Comparison principle for Hamilton-Jacobi-Bellman Equations} \label{ap:Comparison}
	We now state a comparison principle for Hamilton-Jacobi-Bellman equations from \cite{hollender_levy-type_2016} in a simplified way which is suitable for our setting. Let $G:\, \mathbb{R}^d \times \mathbb{R}^d \times \mathbb{S} \times \textnormal{C}_b^2(\mathbb{R}^d) \to \mathbb{R}$ be some operator.
	Set $h(z):=z\,\mathbf{1}_{\{ \|z\| \leq 1\}}(z)$ for all $z \in\mathbb{R}^d$.
	Suppose that there exist an index set $\Theta$ and a family $\{G_{\theta}\}_{\theta \in \Theta}$ of linear operators $G_{\theta}:\,  \mathbb{R}^d \times \mathbb{R}^d \times \mathbb{S} \times \textnormal{C}_b^2(\mathbb{R}^d) \to \mathbb{R}$ given by
	\begin{equation}
		G_{\theta}(x,p,X,f)= -p^T\, b_{\theta}(x)-\tr\Big[X\,a_{\theta}(x)\Big] -\int_{\mathbb{R}^d} \left[ f(x+z)-f(x)-D_xf(x)^T\,h(z) \right] k_{\theta}(dz)\label{eq:defFunctionHollender}
	\end{equation}
	for some Borel-measurable functions $b_\theta,\, a_\theta,\, k_\theta:\,\mathbb{R}^d\rightarrow \mathbb{R}^d,\, \mathbb{S}_+,\, \mathfrak{L}_+$ satisfying Condition \ref{con:HJB}, such that
	\begin{equation} \label{eq:InfAppendix}
		G=\inf_{\theta\in\Theta} G_{\theta}
	\end{equation}
	on $\mathbb{R}^d \times \mathbb{R}^d \times \mathbb{S} \times \textnormal{C}_b^2(\mathbb{R}^d)$. We need a preliminary condition.
	\begin{con} \label{con:HJB}
		Let $\{G_{\theta}\}_{\theta \in \Theta}$ be a family of operators as defined in \eqref{eq:defFunctionHollender} with coefficient functions $b_\theta,\, a_\theta,\, k_\theta$ on $\mathbb{R}^d$.
		Set $\sigma_\theta(x) :=\sqrt{a_\theta(x)}$ for all $x\in\mathbb{R}^d$. \newline
		\begin{longtable}{p{1cm}p{14cm}}
			\emph{(i)} & \emph{Boundedness}: The coefficients of the local part are bounded, i.e.,
			\begin{equation*}
				\sup_{\theta \in \Theta} \Big\{ \| b_{\theta}(x) \| + \| \sigma_{\theta}(x)\| \Big\} < \infty
			\end{equation*}
			in every $x \in \mathbb{R}^d$, and the family of measures of the nonlocal part satisfy
			\begin{equation*}
				\sup_{\theta \in \Theta} \int_{\mathbb{R}^d} \Big( \| z \|^2 \,\mathbf{1}_{ \{\| z \| \leq 1\}}(z) + \mathbf{1}_{ \{\| z \| > 1\}}(z) \Big)\,k_{\theta}(dz) < \infty
			\end{equation*}
			a uniform integrability condition at zero and infinity.\\
			\emph{(ii)} & \emph{Tightness}: The family of measures of the nonlocal part satisfies 
			\begin{equation*}
				\lim_{\delta \to 0} \,\sup_{\theta \in \Theta}\, \int_{\| z \| \leq \delta} \| z \|^2\, k_{\theta}(dz)
				=0
				=\lim_{R \to \infty} \,\sup_{\theta \in \Theta} \,\int_{R < \| z \| }  1\, k_{\theta}(dz)
			\end{equation*}
			 a uniform tightness condition at the origin and infinity.\\
			\emph{(iii)}& \emph{Continuity}: There exists some constant $C>0$ %and a function $\omega: \mathbb{R}_+ \to \mathbb{R}_+$ with $\omega(0)=\lim_{\delta \to 0} \omega(\delta)=0$ 
			such that
			\begin{align*}
				\sup_{\theta \in \Theta} \Big\{ \| b_{\theta}(x) -b_{\theta}(y) \| + \| \sigma_{\theta}(x) -\sigma_{\theta}(y) \| \Big\} \leq C\,\| x-y \|
				%\sup_{\alpha \in \mathcal{A}} \left( \vert f_{\alpha}(t,x) -f_{\alpha}(s,y) \vert + \vert c_{\alpha}(t,x) -c_{\alpha}(s,y) \vert \right) &\leq \omega(\vert t-s \vert )+ \omega(\vert x-y \vert) \\
				%	\sup_{\alpha \in \mathcal{A}} \vert j_{\alpha}(t,x,z)- j_{\alpha}(s,y,z)\vert &\leq \vert z \vert \left( \omega(\vert t- s \vert)+C \vert x-y\vert\right) 
			\end{align*}
			for all %$t,s \in (0,T)$ and 
			$x,y \in \mathbb{R}^d$.
			%\item (Growth Condition) The jump-height coefficients of the nonlocal part have at most linear growth in space, i.e., there exists a constant $C>0$ such that
			%	\begin{equation*}
			%	\sup_{\alpha \in \mathcal{A}} \vert j_{\alpha} (t,x,z)\vert \leq C \left( 1+ \vert x \vert \right) \vert z \vert 
			%	\end{equation*}
			%	for all $t \in (0,T)$ and $x,z \in \mathbb{R}^d$.
			%	\item (Monotonicity) The codomain coefficients are non-positive
			%	\begin{equation*}
			%		c_{\alpha}(t,x) \leq 0
			%	\end{equation*}
			%	for all $t \in (0,T)$ and $x \in \mathbb{R}^d$.
		\end{longtable}
	\end{con}
	\begin{prop} \label{cor:ComparisonPrinciple}
		Let $G$ be an operator satisfying \eqref{eq:InfAppendix} for linear operators $\lbrace G_{\theta} \rbrace_{\theta \in \Theta}$ given in \eqref{eq:defFunctionHollender}. Let $v_1,\, v_2:\,\mathbb{R}_+\times\mathbb{R}^d$ be bounded viscosity sub- and supersolutions respectively of
		\begin{align} \label{eq:AppendixPIDEViscosity}
			\partial_t v(t,x)+ G(x,D_xv(t,x), D_x^2v(t,x), v(t, \cdot)) = 0 
		\end{align}
		in $\mathbb{R}_+\times\mathbb{R}^d$.
		If the initial values $v_1(0,\cdot), v_2(0,\cdot)$ are continuous and 
		\begin{equation*}
			v_1(0,x) \leq v_2(0,x)
		\end{equation*}
		for all $x \in \mathbb{R}^d$, then $v_1(t,x) \leq v_2(t,x)$ for all $(t,x) \in \mathbb{R}_+ \times \mathbb{R}^d$.
	\end{prop}

	\begin{proof}[Proof of Proposition \ref{cor:ComparisonPrinciple}]
		This follows immediately from \cite[Corollary 2.34]{hollender_levy-type_2016}.
	\end{proof}

	\printbibliography
\end{document}